\documentclass[reqno,a4paper,10pt]{amsart}
\usepackage{amsmath}
\usepackage{amsthm}
\usepackage{amssymb}
\usepackage{algorithm}
\usepackage{algorithmic}
\usepackage{graphicx}
\usepackage{grffile}
\usepackage{verbatim}
\usepackage[utf8]{inputenc}
\usepackage[T1]{fontenc}

\usepackage{soul}

\newtheorem{theorem}{\sc{Theorem}}[section]

\newtheorem{lemma}[theorem]{\sc{Lemma}}
\newtheorem{proposition}[theorem]{\sc{Proposition}}

\newtheorem{step}{Step}[theorem]

\theoremstyle{remark}

\theoremstyle{definition}
\newtheorem{definition}{Definition}[section]

\newcommand{\tens}{\otimes}

\usepackage{xspace}
\newcommand{\HE}{(HE)\xspace}
\newcommand{\HEp}{(HE$^+$)\xspace}
\newcommand{\HU}{(HU)\xspace}
\newcommand{\HUp}{(HU$^+$)\xspace}

\newtheorem{remark}{Remark}[section]

\newcommand{\ac}{\mathrm{ac}}
\newcommand{\spt}{\mathrm{spt}}

\DeclareMathOperator{\MA}{MA}
\DeclareMathOperator{\JMA}{JMA}
\DeclareMathOperator{\HMA}{HMA}

\DeclareMathOperator{\DD}{D}
\newcommand{\eps}{\varepsilon}

\renewcommand{\phi}{\varphi}
\renewcommand{\bar}[1]{\overline{#1}}
\renewcommand{\leq}{\leqslant}
\renewcommand{\geq}{\geqslant}
\renewcommand{\div}{\mathrm{div}}

\newcommand{\Prob}{\mathcal{P}}

\newcommand{\B}{\mathrm{B}}

\newcommand{\restr}[2]{\left.#1\right|_{#2}}
\newcommand{\sca}[2]{\langle #1 | #2\rangle}

\newcommand{\Haus}{\mathcal{H}}

\newcommand{\Wass}{\operatorname{W}}

\newcommand{\dd}{\operatorname{d}}

\newcommand{\nr}[1]{\left\Vert #1\right\Vert}
\newcommand{\abs}[1]{\left\vert #1\right\vert}

\renewcommand{\d}{\mathrm{d}}

\newcommand{\LL}{\mathrm{L}}
\newcommand{\Class}{\mathcal{C}}

\DeclareMathOperator{\diam}{diam}

\newcommand{\Rsp}{\mathbb{R}}

\usepackage{xspace}

\newcommand{\const}{\mathrm{const}}

\DeclareMathOperator{\Lag}{Lag}

\newcommand{\mK}{\mathcal{K}}

\newcommand{\mE}{\mathcal{E}}
\newcommand{\mF}{\mathcal{F}}
\newcommand{\mU}{\mathcal{U}}
\newcommand{\eRsp}{\bar{\Rsp}}
\newcommand{\one}{\mathbf{1}}
\usepackage{enumerate}

\author{J.-D. Benamou}
\address{Inria Rocquencourt} 

\author{G. Carlier}
\address{Ceremade, Université Paris-Dauphine} 

\author{Q. Mérigot}
\address{Laboratoire Jean Kuntzmann, Université Grenoble-Alpes / CNRS}

\author{É. Oudet}
\address{Laboratoire Jean Kuntzmann, Université Grenoble-Alpes}

\title{Discretization of functionals involving the Monge-Ampère operator}

\begin{document}

\begin{abstract}
  Gradient flows in the Wasserstein space have become a powerful tool
  in the analysis of diffusion equations, following the seminal work
  of Jordan, Kinderlehrer and Otto (JKO). The numerical applications
  of this formulation have been limited by the difficulty to compute
  the Wasserstein distance in dimension $\geq 2$. One step of the JKO
  scheme is equivalent to a variational problem on the space of convex
  functions, which involves the Monge-Ampère operator. Convexity
  constraints are notably difficult to handle numerically, but in our
  setting the internal energy plays the role of a barrier for these
  constraints. This enables us to introduce a consistent
  discretization, which inherits convexity properties of the
  continuous variational problem. We show the effectiveness of our
  approach on nonlinear diffusion and crowd-motion models.
\end{abstract}

\maketitle

\section{Introduction}

\subsection{Context}
\subsubsection{Optimal transport and displacement convexity}
In the following, we consider two probability measures $\mu$ and $\nu$
on $\Rsp^d$ with finite second moments, the first of which is absolutely continuous with respect
to the Lebesgue measure. We are interested in the quadratic optimal
transport problem between $\mu$ and $\nu$:
\begin{equation}
\min \left\{ \int_{X} \nr{T(x) - x}^2 \dd\mu(x);\, T:X \to \Rsp^d,~T_{\#}\mu = \nu \right\}
\label{eq:ot}
\end{equation}
where $T_{\#}\mu$ denotes the pushforward of $\mu$ by $T$. A theorem
of Brenier shows that the optimal map in \eqref{eq:ot} is given by the
gradient of a convex function \cite{brenier1991polar}. Define the
Wasserstein distance between $\mu$ and $\nu$ as the square root of the
minimum in \eqref{eq:ot}, and denote it $\Wass_2(\mu,\nu)$. Denoting by $\mK$ the 
space of convex functions, Brenier's
theorem implies that for $\phi\in \mK$, 
\begin{equation}
\Wass^2_2(\mu,\nabla\phi_{\#}\mu) = \int_{\Rsp^d} \nr{x - \nabla\phi(x)}^2 \dd\mu(x)
\end{equation}
and that the map defined 
\begin{equation}
\phi \in \mK \mapsto \nabla\phi_{\#}\mu \in \Prob(\Rsp^d),
\label{eq:brenier}
\end{equation}
is onto. This map can be seen as a parameterization, depending on
$\mu$, of the space of probability measures by the set of convex potentials $\mK$.
 This idea has been exploited by McCann
\cite{mccann1997convexity} to study the steady states of gases whose
energy $\mF: \Prob(\Rsp^d)\to\Rsp$ is the sum of an internal energy
$\mU$, such as the negative entropy, and an interaction energy
$\mE$. McCann gave sufficient conditions for such a functional $\mF$
to be convex along minimizing Wasserstein geodesics. These conditions
actually imply a stronger convexity property for the functional $\mF$:
this functional is convex under generalized displacement: for any
absolutely continuous probability measure $\mu$, the composition of
$\mF$ with the parameterization given in Eq.~\eqref{eq:brenier}, $\phi
\in \mK \mapsto \mF(\nabla\phi_{\# \mu})$, is convex. Generalized
displacement convexity allows one to turn a non-convex optimization
problem over the space of probability measures into a convex
optimization problem on the space of convex functions.

\subsubsection{Gradient flows in Wasserstein space and JKO scheme}%
Our goal is to simulate numerically non-linear evolution PDEs which
can be formulated as gradient flows in the Wasserstein space. The
first formulation of this type has been introduced in the seminal
article of Jordan, Kinderlehrer and Otto
\cite{jordan1998variational}. The authors considered the linear
Fokker-Planck equation
\begin{equation}
\left\{
\begin{aligned}
&\frac{\partial\rho}{\partial t} = \Delta \rho + \div(\rho \nabla V) \\
&\rho(0,.) = \rho_0
\end{aligned},
\right.
\label{eq:fk}
\end{equation}
where $\rho(t,.)$ is a time-varying probability density on $\Rsp^d$
and $V$ is a potential energy. The main result of the article is that
\eqref{eq:fk} can be reinterpreted as the gradient flow in the
Wasserstein space of the energy functional
\begin{equation}
\mF(\rho) = \int_{\Rsp^d} (\log
\rho(x) + V(x)) \rho(x) \dd x.
\end{equation}
Jordan, Kinderlehrer and Otto showed how to construct such a gradient
flow through a time-discretization, using a generalization of the
backward Euler scheme. Given a timestep $\tau$, one defines
recursively a sequence of probability densities $(\rho_k)_{k\geq 0}$ :
\begin{equation}
\rho_{k+1} = \arg\min_{\rho\in \Prob^\ac(\Rsp^d)} \frac{1}{2\tau} \Wass_2^2(\rho_k,\rho) + \mF(\rho).
\label{eq:jko}
\end{equation}
The main theorem of \cite{jordan1998variational} is that the discrete
gradient flow constructed by \eqref{eq:jko} converges to the solution
of the Fokker-Planck equation \eqref{eq:fk} in a suitable weak sense
as $\tau$ tends to zero.  Similar formulations have been proposed for
other non-linear partial differential equations : the porous medium
equation \cite{otto2001geometry} and more general degenerate parabolic
PDEs \cite{agueh2005existence}, the sub-critical Keller-Segel equation
\cite{blanchet2008convergence}, macroscopic models of crowds
\cite{maury2010macroscopic}, to name but a few. The construction and
properties of gradient flows in the Wasserstein space have been
studied systematically in \cite{ambrosio2005gradient}. Finally, even solving for a single step
 of the JKO scheme leads to nontrivial nonlocal PDEs of Monge-Amp\`ere
type which appear for instance in the Cournot-Nash problem in game theory \cite{blanchet2012optimal}.

\subsection{Previous work}
\subsubsection{Numerical resolution of gradient flows.} Despite the
potential applications, there exists very few numerical simulations
that use the Jordan-Kinderlehrer-Otto scheme and its
generalizations. The main reason is that the first term of the
functional that one needs to minimize at each time step, e.g.
Eq. \eqref{eq:jko}, is the Wasserstein distance. Computing the
Wasserstein distance and its gradient is notably difficult in
dimension two or more. In dimension one however, the optimal transport
problem is much simpler because of its relation to monotone
rearrangement. This remark has been used to implement discrete
gradient flows for the quadratic cost
\cite{kinderlehrer1999approximation,blanchet2008convergence,
  blanchet2012optimal} or for more general convex costs
\cite{agueh2013one}. In $2$D, the Lagragian method proposed in
\cite{carrillo2009numerical,burger2010mixed} is inspired by the JKO
formulation but the convexity of the potential is not enforced.

\subsubsection{Calculus of variation under convexity constraints.}
When the functional $\mF$ is convex under generalized displacement,
one can use the parameterization Eq.~\eqref{eq:brenier} to transform
the problem into a convex optimization problem over the space of
convex functions. Optimization problems over the space of convex
functions are also frequent in economy and geometry, and have been
studied extensively, from a numerical viewpoint, when $\mF$ is an
integral functional that involve function values and gradients:
\begin{equation}
\min_{\phi\in\mK} \int_{\Omega} F(x,\phi(x),\nabla\phi(x)) \dd x
\label{eq:cvx}
\end{equation}

The main difficulty to solve this minimization problem numerically is
to construct a suitable discretization of the space of convex
functions over $\Omega$. The first approach that has been considered
is to approximate $\mK$ by piecewise linear functions over a fixed
mesh. This approach has an important advantage: the number of linear
constraints needed to ensure that a piecewise linear function over a
mesh is convex is proportional to the size of the mesh. Unfortunately,
Choné and Le Meur \cite{chone2001non} showed that there exists convex
functions on the unit square that cannot be approximated by
piecewise-linear convex functions on the regular grid with edgelength
$\delta$, even as $\delta$ converges to zero. This difficulty has
generated an important amount of research in the last decade.

Finite difference approaches have been proposed by Carlier,
Lachand-Robert and Maury \cite{carlier2001numerical}, based on the
notion of convex interpolate, and by Ekeland and Moreno-Bromberg using
the representation of a convex function as a maximum of affine
functions \cite{ekeland2010algorithm}, taking inspiration from Oudet
and Lachand-Robert \cite{lachand2005minimizing}.  In both methods, the
number of linear inequality constraints used to discretize the
convexity constraints is quadratic in the number of input points, thus
limiting the applicability of these methods. More recently, Mirebeau
proposed a refinement of these methods in which the set of
active constraints is learned during the optimization process
\cite{mirebeau2013}. Oberman used the idea of imposing convexity
constraints on a wide-stencils \cite{oberman2013numerical}, which
amounts to only selecting the constraints that involve nearby points
in the formulation of \cite{carlier2001numerical}. Oudet and Mérigot
\cite{merigot2014handling} used interpolation operators to approximate
the solutions of \eqref{eq:cvx} on more general finite-dimensional
spaces of functions. All these methods can be used to minimize
functionals that involve the value of the function and its gradient
only. They are not able to handle terms that involve the Monge-Ampère
operator $\det\DD^2\phi$ of the function, which appears when e.g.
considering the negative entropy of $\restr{\nabla\phi}{\#} \rho$. It is worth mentioning here that 
convex variational problems with a convexity constraint and involving the Monge-Ampère
operator $\det\DD^2\phi$ appear naturally in geometric problems such as the affine Plateau problem, 
see Trudinger and Wang \cite{trudinger2005affine} or Abreu's equation, see Zhou \cite{zhou2012first}. 
The Euler-Lagrange equations of such problems are fully nonlinear fourth-order PDEs and looking 
numerically for convex solutions can be done by similar methods as the ones developed in the present paper.  

\subsection{Contributions.}

In this article, we construct a discretization in space of the type
of variational problems that appear in the definition of the JKO
scheme. More precisely, given two bounded convex subsets $X,Y$ of
$\Rsp^d$, and an absolutely continuous measure $\mu$ on $X$, we want
to discretize in space the minimization problem
\begin{equation}
\min_{\nu \in \Prob(Y)} \Wass^2_2(\mu,\nu) + \mE(\nu) + \mU(\nu),
\label{eq:prob-orig}
\end{equation}
where $\Prob(Y)$ denotes the set of probability measures on $Y$, and
where the potential energy $\mE$ and the internal energy $\mU$ are
defined as follows:
\begin{align}
  \mE(\nu) &= \int_{\Rsp^d} \int_{\Rsp^d} W(x,y) \dd[\nu\tens\nu](x,y) + \int_{\Rsp^d} V(x) \dd\nu(x)\\
  \mU(\nu) &= \left\{
\begin{aligned}
&\int_{\Rsp^d} U(\sigma(x))\dd x \hbox{ if } \dd \nu = \sigma \dd \Haus^d, \sigma \in \LL^1(\Rsp^d)\\
& +\infty \hbox{ if  not}
\end{aligned}\right.
\label{eq:U}
\end{align}
We assume McCann's sufficient conditions \cite{mccann1997convexity}
for the generalized displacement convexity of the functional $\mF =
\Wass_2^2(\mu,.) + \mE + \mU$, namely:
\begin{itemize}
\item[\HE] the potential $V: \Rsp^d\to\Rsp$ and interaction potential
  $W:\Rsp^d\times\Rsp^d\to\Rsp$ are convex functions. (If in addition
  $V$ or $W$ is strictly convex, we  denote this assumption \HEp)
\item[\HU] The function $U:\Rsp^+\to\Rsp$ is such that the map
  $r\mapsto r^d U(r^{-d})$ is convex and non-increasing, and $U(0) =
  0$. (If the convexity of $r\mapsto r^d U(r^{-d})$ is strict, we denote 
  this assumption \HUp.)
\end{itemize}
Under assumptions \HE and \HU, the problem \eqref{eq:prob-orig} can be
rewritten as a convex optimization problem. Introducing the space
$\mK_Y$ of convex functions on $\Rsp^d$ whose gradient lie in $Y$
almost everywhere, \eqref{eq:prob-orig} is equivalent to
\begin{equation}
\min_{\phi \in \mK_Y} \Wass^2_2(\mu,\nabla\phi_{\#}\mu) + \mE(\nabla\phi_{\#}\mu) + \mU(\nabla\phi_{\#}\mu).
\label{eq:prob}
\end{equation}
Our contributions are the following:
\begin{itemize}
\item In Section~\ref{sec:discretization}, we discretize the space
  $\mK_Y$ of convex functions with gradients contained in $Y$ by
  associating to every finite subset $P$ of $\Rsp^d$ a
  finite-dimensional convex subset $\mK_Y(P)$ contained in the space
  of real-valued functions on the finite-set $P$. We construct a
  discrete Monge-Ampère operator, in the spirit of Alexandrov, which
  satisfies some structural properties of the operator $\phi \mapsto
  \det(\DD^2\phi)$, such as Minkowski's determinant
  inequality. Moreover, we show how to modify the construction of
  $\mK_Y(P)$ so as to get a linear gradient operator, following an
  idea of Ekeland and Moreno-Bromberg \cite{ekeland2010algorithm}.
\item In Section~\ref{sec:convexity}, we construct a convex
  discretization of the problem \eqref{eq:prob}. In order to do so, we
  need to define an analogous of $\nabla\phi_{\#} \mu$, where $\phi$
  is a function in our discrete space $\mK_Y(P)$ and where $\mu_P$ is
  a measure supported on $P$. It turns out that in order to maintain
  the convexity of the discrete problem, one needs to define two such
  notions: the pushforward $G_{\phi\#}^\ac \mu_P$ which is absolutely
  continous on $Y$ and whose construction involves the discrete
  Monge-Ampère operator, and $G_{\phi\#} \mu_P$ which is supported on
  a finite set and whose construction involves the discrete
  gradient. The discretization of \eqref{eq:prob} is given by
\begin{equation}
\min_{\phi \in \mK_Y(P)} \Wass^2_2(\mu,G_{\phi\#}\mu_P) + \mE(G_{\phi\#}\mu_P) + \mU(G^\ac_{\phi\#}\mu_P).
\label{eq:prob:disc}
\end{equation}
\item In Section~\ref{sec:convergence}, we show that if
  $(\mu_{P_n})_{n\geq 0}$ is a sequence of probability measures on $X$
  that converge to $\mu$ in the Wasserstein sense, minimizers of the
  discretized problem \eqref{eq:prob:disc} with $P=P_n$ converge, in a
  sense to be made precise, to minimizers of the continuous
  problem. In order to prove this result, we need a few additional
  assumptions: the density of $\mu$ should be bounded from above and
  below on the convex domain $X$, and the integrand in the definition
  of the internal energy \eqref{eq:U} should be convex.
\item Finally, in Section~\ref{sec:numerical} we present two numerical
  applications of the space-discretization \eqref{eq:prob:disc}. Our
  first simulation is a meshless Lagrangian simulation of the porous
  medium equation and the fast-diffusion equation using the gradient
  flow formulation of Otto \cite{otto2001geometry}. The second
  simulation concerns the gradient-flow model of crowd motion
  introduced by Maury, Roudneff-Chupin and Santambrogio
  \cite{maury2010macroscopic}.
\end{itemize}

\subsection*{Notation} 
The Lebesgue measure is denoted $\Haus^d$.  The space of probability
measures on a domain $X$ of $\Rsp^d$ is denoted $\Prob(X)$, while
$\Prob^\ac(X)$ denotes the space of probability measures that are
absolutely continuous with respect to the Lebesgue measure.

\section{Discretization of the space of convex functions}
\label{sec:discretization}

The first goal of this section is to discretize the
space of convex functions whose gradients lie in a prescribed convex
set $Y$. Then, we will define a notion of discrete Monge-Ampère operator for
functions in this space. We will consider functions from $\Rsp^d$ to
the set of extended reals $\eRsp := \Rsp\cup \{+\infty\}$.

\begin{definition}[Legendre-Fenchel transform] The Legendre transform
  $\psi^*$ of a function $\psi: Y \to\eRsp$, is the function
  $\psi^*:\Rsp^d \to\eRsp$ defined by the formula
  \begin{equation} \psi^*(x) := \sup_{y\in Y} \sca{x}{y} -
    \psi(y).\end{equation} The space of Legendre-Fenchel transforms of
  functions defined over a convex set $Y$ is denoted by $\mK_Y := \{
  \psi^*; \psi:Y\to\eRsp\}.$ A function on $\Rsp^d$ is called
  \emph{trivial} if it is constant and equal to $+\infty$. The space
  of non-trivial functions in $\mK_Y$ is denoted $\mK_Y^0$.
\label{def:legendre}
\end{definition}

\begin{lemma} Assume that $Y$ is a bounded convex subset of
  $\Rsp^d$. Then,
\begin{enumerate}[(i)]
\item functions in $\mK_Y$ are trivial or finite everywhere: $\mK_Y =
  \mK_Y^0 \cup \{+\infty\}$;
\item a convex function $\phi$ belongs to $\Class^1 \cap \mK_Y^0$ if
  and only if $\nabla\phi(\Rsp^d) \subseteq Y$ ;
\item the set $\Class^1 \cap \mK_Y^0$ is dense in the set
  $\mK_Y^0$ for $\nr{.}_\infty$;
\item the space $\mK_Y$ is convex;
\item (stability by maximum) given a family of functions
  $(\phi_i)_{i\in I}$ in $\mK_Y$, the function $\phi(x) := \sup_{i\in
    I} \phi_i(x)$ is also in $\mK_Y$.
\end{enumerate}
\end{lemma}

\begin{proof}
  (i) We assume that $\phi$ belongs to $\mK_Y$, i.e. $\phi = \psi^*$, where
  $\psi$ is a function from $Y$ to $\bar{\Rsp}$. We will first show
  that if $\phi$ is non-trivial, then $\psi$ is lower bounded by a
  constant on $Y$. By contradiction, assume that there exists a set of
  points $y_k$ in $Y$ such that $\psi(y_k) \to -\infty$. In this case,
  given any point $x$ in $X$ we have
\begin{equation*} \psi^*(x) \geq \max_k \sca{x}{y_k} - \psi(y_k) \geq 
\max_k - \nr{x}\nr{y_k} - \psi(y_k) = +\infty,
\end{equation*}
so that $\phi$ is trivial.

(iii) Assume that $\phi$ belongs to $\mK_Y^0$, so that there exists a
convex function $\psi: Y\to\Rsp$ lower bounded by a constant and such that
$\psi^* = \phi$. Then, we can approximate $\psi$ by uniformly convex
functions $\psi_\eps(y) := \psi(y) + \eps\nr{y}^2$ on $Y$. The
functions $\phi_\eps := \psi_\eps^*$ belong to $\mK_Y$, are smooth,
and uniformly converge to the function $\phi$.
\end{proof}

\begin{definition}[$\mK_Y$-envelope and $\mK_Y$-interpolate]
  The \emph{$\mK_Y$--envelope} of a function $\phi$ defined on a
  subset $P$ of $\Rsp^d$ is the largest function in $\mK_Y$ whose
  restriction to $P$ lies below $\phi$. In other words,
\begin{equation}
  \phi_{\mK_Y} := \max \{ \psi \in \mK_Y; \restr{\psi}{P} \leq \restr{\phi}{P}\}.
\end{equation}
A function $\phi$ on a set $P\subseteq \Rsp^d$ is a
\emph{$\mK_Y$-interpolate} if it coincides with the restriction to $P$
of its $\mK_Y$--envelope. The space of $\mK_Y$--interpolates is denoted
\begin{equation}
  \mK_Y(P) := \{ \phi:P \to \Rsp; \phi = \restr{\phi_{\mK_Y}}{P} \}.
\end{equation}
\end{definition}

\subsection{Subdifferential and Laguerre cells}
Consider a convex function $\phi$ on $\Rsp^d$, and a point $x$. A
vector $y\in\Rsp^d$ is a \emph{subgradient} of $\phi$ at $x$ if for
every $z$ in $\Rsp^d$, the inequality $\phi(z) \geq \phi(x) +
\sca{z-x}{y}$ holds. The \emph{subdifferential} of $\phi$ at $x$ is
the set of subgradients to $\phi$ at $x$, i.e.
\begin{equation}
\partial \phi(x) := \{ y \in \Rsp^d; \forall z\in \Rsp^d,\phi(z) \geq \phi(x) + \sca{z-x}{y} \}
\end{equation}
The following lemma allows one to compute the subdifferential of the
$\mK_Y$--envelope of a function in $\mK_Y(P)$.

\begin{definition}[Laguerre cell] Given a finite point set $P$
  contained in $\Rsp^d$, a function $\phi$ on $P$, we denote the
  \emph{Laguerre cell} of a point $p$ in $P$ the polyhedron
$$\Lag_P^\phi(p) := \{ y \in \Rsp^d; \forall q\in P,\phi(q) \geq \phi(p) + \sca{q-p}{y} \}.$$
Note that the union of the Laguerre cells covers the space, while the
intersection of the interior of two Laguerre cells is always empty.
\end{definition}

\begin{lemma} Let $P$ be a finite point set. A function $\phi$ on $P$
  belongs to $\mK_Y(P)$ if and only if for every $p$ in $P$, the
  intersection $\Lag_P^\phi(p) \cap Y$ is non-empty. Moreover, if this
  is the case, then 
\begin{equation}
\partial \phi_{\mK_Y}(p) = \Lag_P^\phi(p) \cap Y.
\label{eq:convex}
\end{equation}
\end{lemma}
\begin{proof} Denote $\mK := \mK_{\Rsp^d}$ and $\phi_{\mK}$ the convex
  envelope of $\phi$. It is then easy to see that for every point $p$
  in $P$ such that $\phi_{\mK}(p) = \phi(p)$,
$$\partial \phi_{\mK}(p) = \Lag_P^\phi(p).$$
Since $\mK_Y \subseteq \mK$ and by definition, one has
$\phi_{\mK_Y}(x) \leq \phi_{\mK}(x)$, with equality when $x$ is a
point in $P$. This implies the inclusion $\partial \phi_{\mK_Y}(p)
\subseteq Y \cap \partial \phi_{\mK}(p).$ In order to show that the
converse also holds, one only needs to remark that
\begin{equation*} Y \subseteq \bigcup_{p \in P} \partial \phi_{\mK_Y}(p).\qedhere
\end{equation*}
\end{proof}

\begin{lemma}
Let $\phi_0, \phi_1$ in $\mK_Y(P)$, let $\phi_t =
  (1-t)\phi_0 + t\phi_1$ be the linear interpolation on $P$ between
  these functions, and denote $\hat{\phi}_t := [\phi_t]_{\mK_Y}$. Then
  for any $p$ in $P$, 
\label{lem:convex}
\begin{align}
\partial \hat{\phi}_t (p) &\supseteq (1-t) \partial
  \hat{\phi}_0(p) + t \partial
  \hat{\phi}_1(p)\\
\Lag_P^{\phi_t}(p)\cap Y &\supseteq 
(1-t) (\Lag_P^{\phi_0}(p)\cap Y) + 
t (\Lag_P^{\phi_1}(p)\cap Y)
\end{align}
\end{lemma}

\begin{proof}
  Thanks to the previous lemma, the two inclusions are
  equivalent. Now, let $y_i$ be a point in
$\Lag_P^{\phi_i}(p)\cap Y$, so that
$$ \forall q\in P,~ \phi_i(q) \geq \phi_i(p) + \sca{q-p}{y_i}.$$
Taking a linear combination of these inequalities, we get
$$ \forall q\in P,~ (1-t) \phi_0(q) + t\phi_1(q) \geq (1-t) \phi_0(p) + t\phi_1(p) + \sca{q-p}{y_t},$$
with $y_t = (1-t) y_0 + ty_1$. In other words, the point $y_t$ belongs
to the Laguerre cell $\Lag_P^{\phi_t}(p)$. Since this holds for any
pair of points $y_0$ in $\Lag_P^{\phi_0}(p)$ and $y_1$ in
$\Lag_P^{\phi_1}(p)$, we get the desired inclusion.
\end{proof}

\begin{remark} A corollary of the two previous lemmas is the convexity
  of the space $\mK_Y(P)$ of $\mK_Y$-interpolates, a fact that does
  not obviously follow from the definition. 
\end{remark}
\begin{remark} The convex envelope of a function defined on a finite
  set is always piecewise-linear. In contrast, when the domain $Y$ is
  bounded, the $\mK_Y$-envelope of an element $\phi$ of the polyhedron
  $\mK_Y(P)$ does not need to be piecewise linear, even when
  restricted to the convex hull of $P$. Fortunately, for the
  applications that we are targeting, we will never need to compute
  this envelope explicitely, and we will only use formula
  \eqref{eq:convex} giving the explicit expression of the
  subdifferential.
\end{remark}

\subsection{Monge-Ampère operator} In this paragraph, we introduce a
notion of discrete Monge-Ampère operator of $\mK_Y$-interpolates on a
finite set.  This definition is closely related to the notion of
Monge-Ampère measure introduced by Alexandrov. Given a smooth
uniformly convex function $\phi$ on $\Rsp^d$, a change of variable
gives
\begin{equation}
\int_B \det(\DD^2\phi(x))\dd x = \int_{\nabla\phi(B)} 1\dd x = \Haus^d(\nabla\phi(B)).
\end{equation}
This equation allows one to define a measure on the source domain $X
\subseteq \Rsp^d$, called the Monge-Ampère measure and denoted
$\MA[\phi]$. Using the right-hand side of the equality, it is possible
to extend the notion of Monge-Ampère measure to convex functions that
are not necessarily smooth (see e.g.  \cite{gutierrez2001monge}):
\begin{equation}
  \MA[\phi](B) :=
  \Haus^d(\partial\phi(B)).
\end{equation}

\begin{definition}
  The discrete Monge-Ampère operator of a $\mK_Y$-interpolate
  $\phi:P\to\Rsp$ at a point $p$ in $P$ is defined by the formula:
\begin{equation}
  \MA_Y[\phi](p) := \Haus^d(\partial \phi_{\mK_{Y}}(p)),
\end{equation}
where $\Haus^d$ denotes the $d$-dimensional Lebesgue measure.
\end{definition}

The relation between the discrete Monge-Ampère operator and the
Monge-Ampère measure is given by the formula:
\begin{equation}
\forall \phi\in\mK_Y(P),~ \MA[\phi_{\mK_Y}] = \sum_{p\in P} \MA_Y[\phi](p) \delta_p.
\end{equation}
In other words, the Monge-Ampère operator can be seen as the density
of the Monge-Ampère measure of $\phi_{\mK_Y}$ with respect to the
counting measure on $P$. The next lemma is crucial to the proof of
convexity of our discretized energies. It is also interesting in
itself, as it shows that the interior of the set $\mK_Y(P)$ of convex
interpolates can be defined by $\abs{P}$ explicit non-linear convex
constraints.


\begin{lemma} For any point $p$ in $P$, the following map is convex:
\begin{equation}
\phi \in \mK_Y(P) \mapsto -\log(\MA_{Y}[\phi](p)).
\end{equation}
\label{lemma:logconcavity}
\end{lemma}
\begin{proof} 
Let $\phi_0, \phi_1$ in $\mK_Y(P)$, let $\phi_t =
  (1-t)\phi_0 + t\phi_1$ be the linear interpolation between these
  functions, and denote $\hat{\phi}_t := [\phi_t]_{\mK_Y}$. 
%
  Using Lemma~\ref{lem:convex}, and with the convention $\log(0) =
  -\infty$, we have
\begin{align*}
\log(\Haus^d(\partial \hat{\phi_t}(p)))
&\geq \log(\Haus^d((1-t) \partial \hat{\phi}_0(p) + t \partial   \hat{\phi}_1(p))) \\
&\geq (1-t) \log(\Haus^d(\partial \hat{\phi}_0(p))) + t \log(\Haus^d(\partial   \hat{\phi}_1(p))), 
\end{align*}
where the second inequality is the logarithmic version of the
Brunn-Minkowski inequality. 
\end{proof}

\subsection{Convex interpolate with gradient}%
In applications, we want to minimize energy functionals over the space
$\mK_Y$, which involve potential energy terms such as \begin{equation}
\phi \mapsto
\int_{X} V(\nabla \phi(x))\dd\mu(x),
\end{equation}
where $V$ is a convex potential on $\Rsp^d$. Any functional defined
this way is convex in $\phi$, and one would like to be able to define
a discretization of this functionals that preserves this
property. Given a function $\phi$ in the space $\mK_Y(P)$ and a point
$p$ in $P$, one wants to select a vector in the subdifferential
$\partial \phi_{\mK_Y}(p)$, and this vector needs to depend linearly
on $\phi$. A way to achieve this is to increase the dimension of the
space of variables, and to include the chosen subgradients as unknown
of the problem. This can be done as in Ekeland and Moreno-Bromberg
\cite{ekeland2010algorithm}.

\begin{definition}[Convex interpolate with gradient]
  A \emph{$\mK_Y$-interpolate with gradient} on a finite subset $P$ of
  $\Rsp^d$ is a couple $(\phi, G_\phi)$ consisting of a function
  $\phi$ in the space of $\mK_Y$-interpolates $\mK_Y(P)$ and a
  gradient map $G_\phi: P\to\Rsp^d$ such that 
\begin{equation}
\forall p\in P,~ G_\phi(p) \in \partial \phi_{\mK_Y}(p).
\end{equation} The space of convex interpolates with gradients is
denoted $\mK_Y^G(P)$.
\end{definition}

Note that the space $\mK_Y^G(P)$ can be considered as a subset of the
vector space of function from $P$ to $\Rsp\times
\Rsp^d$. Lemma~\ref{lemma:dist} below implies that $\mK_Y^G(P)$ forms
a convex subset of this vector space, for which one can construct
explicit convex barriers.
Given a closed subset $A$ of $\Rsp^d$ and $x$ a point of $\Rsp^d$,
$\dd(x, A)$ denotes the minimum distance between $x$ and any point in
$A$.

\begin{lemma} \label{lemma:dist} Let $\phi_0$ and $\phi_1$ be two
  functions in $\mK_Y(P)$ and let $v_i$ a vector in the
  subdifferential $\partial \hat{\phi}_i(p)$ for a certain point $p$
  in $P$. Then,
\begin{enumerate}[(i)]
\item the vector $v_t = (1-t) v_0 + t v_1$ lies in $\partial \hat{\phi}_t(p)$; 
\item the map $t \mapsto \d(v_t, \Rsp^d\setminus \partial \hat{\phi}_t(p))$ is concave; 
\end{enumerate}
Moreover, a function $\phi$ belongs to the interior of $\mK_Y(P)$
if and only if
\begin{equation}
  \forall p\in P,~ \MA_{Y}[\phi](p) > 0.
\end{equation}
\end{lemma}

\begin{proof}
  The first item is a simple consequence of Lemma~\ref{lem:convex}. In
  order to prove the second item, we first remark that setting $R_i :=
  \d(v_i, \Rsp^d\setminus \partial \hat{\phi}_i(p))$, one has:
  $\B(v_i,R_i) \subseteq \partial \hat{\phi}_i(p)$. Using the second
  inclusion from Lemma~\ref{lem:convex} and the explicit formula for
  the Minkowski sum of balls, we get:
  $$(1-t) \B(v_0,R_0) + t \B(v_1,R_1) = \B(v_t, (1-t) R_0 + t R_1) \subseteq \partial \hat{\phi}_t(p).$$
  This implies the desired concavity property:
\begin{equation*}
\dd(v_t, \Rsp^d \setminus \partial \hat{\phi}_t(p)) \geq (1-t) R_0 + t R_1.
\end{equation*}
As for the last assertion, assume by contradiction that there is a
$p\in P$ such that $\partial \phi_{\mK_Y}(p)$ has empty interior, and
let $y\in \partial \phi_{\mK_Y}(p)$. Since the Laguerre cells cover
the space, this means that $y$ also belongs to (the boundary) of
Laguerre cells with nonempty interior corresponding to points $p_1,
\dots, p_k\in P^k$ for some $k\ge 2$. In this case necessarily, $p$ is
in the relative interior of the convex hull of $\{p_1, \dots, p_k\}$
and $ \phi_{\mK_Y}$ is affine on this convex hull, contradicting
interiority of $\phi$.
\end{proof}

\section{Convex discretization of displacement-convex functionals}
\label{sec:convexity}

In the discrete setting, the reference probability density $\rho$ is
replaced by a probability measure $\mu$ on a finite point set. Since
the subdifferential of a convex function $\phi$ can be multi-valued,
the pushforward $\nabla \phi_{\#}\mu$ is not uniquely defined in
general. In order to maintain the convexity properties of the three
functionals in our discrete setting, we will need to consider two
different type of push-forwards.

\begin{definition}[Push-forwards] Let $\mu$ be a probability measure
  supported on a finite point set $P$, i.e. $\mu = \sum_{p \in P}
  \mu_p \delta_p$. We consider a convex interpolate with gradient
  $(\phi,G_\phi)$ in $\mK_Y^G(P)$, and we define two ways of pushing
  forward the measure $\mu$ by the gradient of $\phi_{\mK_Y}$.
\begin{itemize}
\item The first way consists in moving each Dirac mass $\mu_p\delta_p$
  to the selected subgradient $G_\phi(p)$, thus defining
\begin{equation} G_{\phi\#}\mu := \sum_{p \in P} \mu_p \delta_{G_\phi(p)}.
\end{equation}
\item The second possibility, is to spread each Dirac mass
  $\mu_p\delta_p$ on the whole subdifferential $\partial
  \phi_{\mK_Y}(p)$. This defines, when $\phi$ is in the interior of $\mK_Y(P)$,   
 an absolutely continuous measure:
\begin{equation} G^\ac_{\phi\#} \mu := \sum_{p \in P} \mu_p
  \frac{\restr{\Haus^d}{\partial \phi_{\mK_Y}(p)}}{\Haus^d(\partial
    \phi_{\mK_Y}(p))}.
\end{equation}
\end{itemize}
\end{definition}
\begin{remark}
  Note that in both cases, the mass of $\mu$ located at $p$ is
  transported into the subdifferential $\partial\phi_{\mK_Y}(p)$. This
  implies that the transport plan between $\mu$ and $G_{\phi\#}\mu$
  induced by this definition is optimal, and similarly for
  $G_{\phi\#}^\ac\mu$. We therefore have an explicit expression for
  the squared Wasserstein distance between $\mu$ and these
  pushforwards:
\begin{align}
  \Wass^2_2(\mu,G_{\phi\#}\mu) &= \sum_{p\in P} \mu_p \nr{p - G_\phi(p)}^2  \\
  \Wass^2_2(\mu,G^\ac_{\phi\#}\mu) &= \sum_{p\in P}
  \frac{\mu_p}{\Haus^d(\partial \phi_{\mK_Y}(p))} \int_{\partial
      \phi_{\mK_Y}(p)} \nr{p - x}^2 \dd x
\end{align}
\end{remark}

\begin{theorem} Given a bounded convex set $Y$ and a measure $\mu$
  supported on a finite set $P$, and under hypothesis \HE and \HU,
  the maps \label{th:convex}
\begin{align}
 (\phi,G_\phi) \in \mK^G_Y(P) &\mapsto \mE(G_{\phi\#} \mu) \label{eq:conv:E}\\
 \phi \in \mK_Y(P) &\mapsto \mU(G^\ac_{\phi\#} \mu) \label{eq:conv:U}
\end{align}
are convex. Moreover, under assumptions \HEp and \HUp the functional
\begin{equation} (\phi,G_\phi) \in \mK^G_Y(P) \mapsto \mF(\phi) := \mE(G_{\phi\#} \mu) + \mU(G^\ac_{\phi\#} \mu)
\end{equation}
has the following strict convexity property: given two functions
$\phi_0$ and $\phi_1$ in $\mK_Y^G(P)$, and $\phi_t = (1-t)\phi_0 +
t\phi_1$ with $t\in (0,1)$, then
$$ \mF(\phi_t) \leq (1-t) \mF(\phi_0) + t \mF(\phi_1), $$
with equality only if $\phi_0 - \phi_1$ is a constant. In particular,
there is at most one minimizer of $\mF$  up to an additive constant.
\end{theorem}

\begin{proof} The proof of \eqref{eq:conv:U} uses the log-concavity of
  the discrete Monge-Ampère operator as in
  Lemma~\ref{lemma:logconcavity} and McCann's condition
  \cite{mccann1997convexity}. The proof of \eqref{eq:conv:E} is
  direct: if $(\phi_0,G_{\phi_0})$ and $(\phi_0,G_{\phi_0})$ belong to
  $\mK^G_Y(P)$, and $G_{\phi_t} := (1-t)G_{\phi_0} + t G_{\phi_1}$,
  then the convexity of $\mE$ follows from that of $V$ and $W$.
\end{proof}

\begin{figure}[t]
{
\centering
\includegraphics[height=3.6cm]{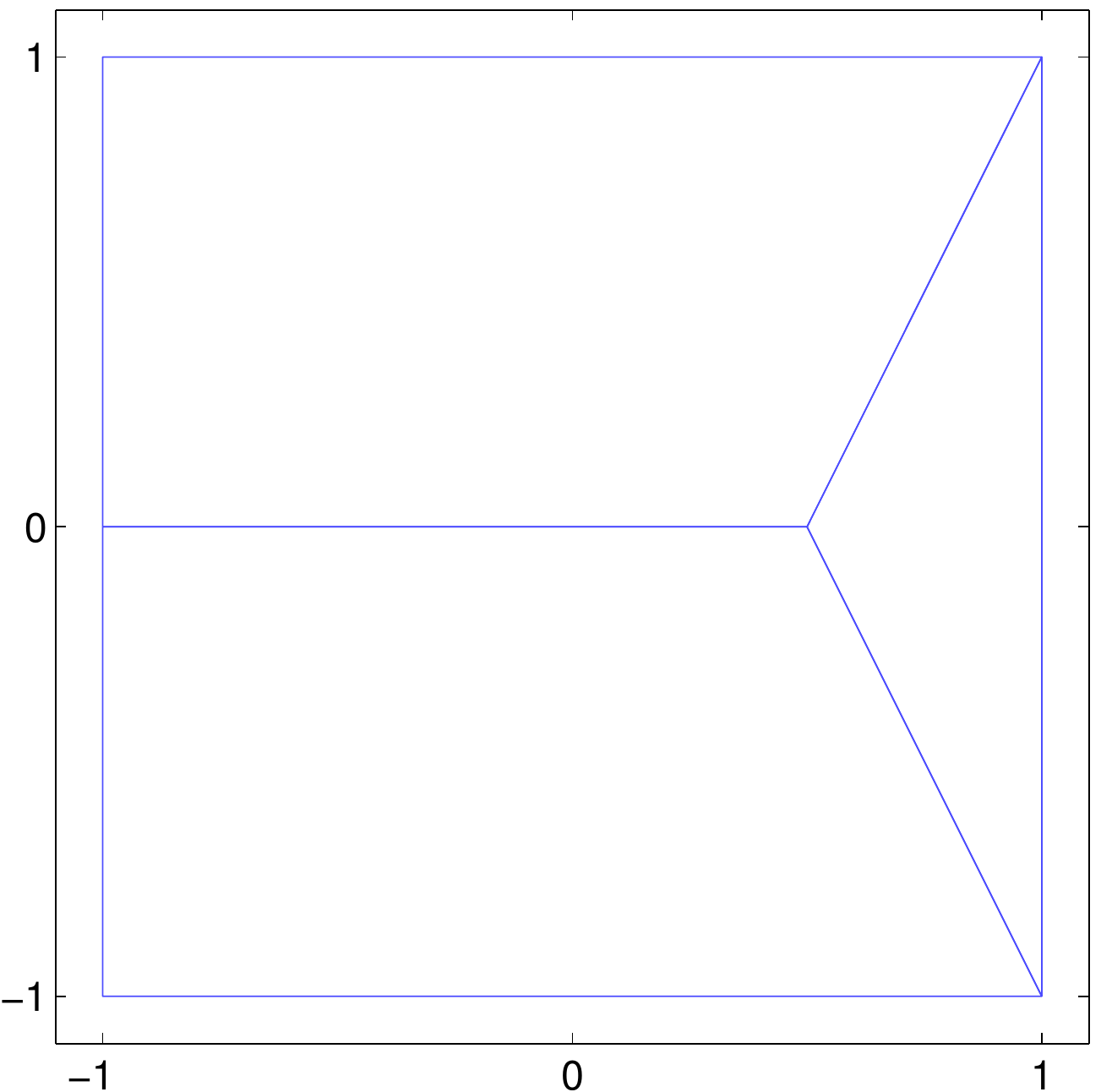}\,
\includegraphics[height=3.6cm]{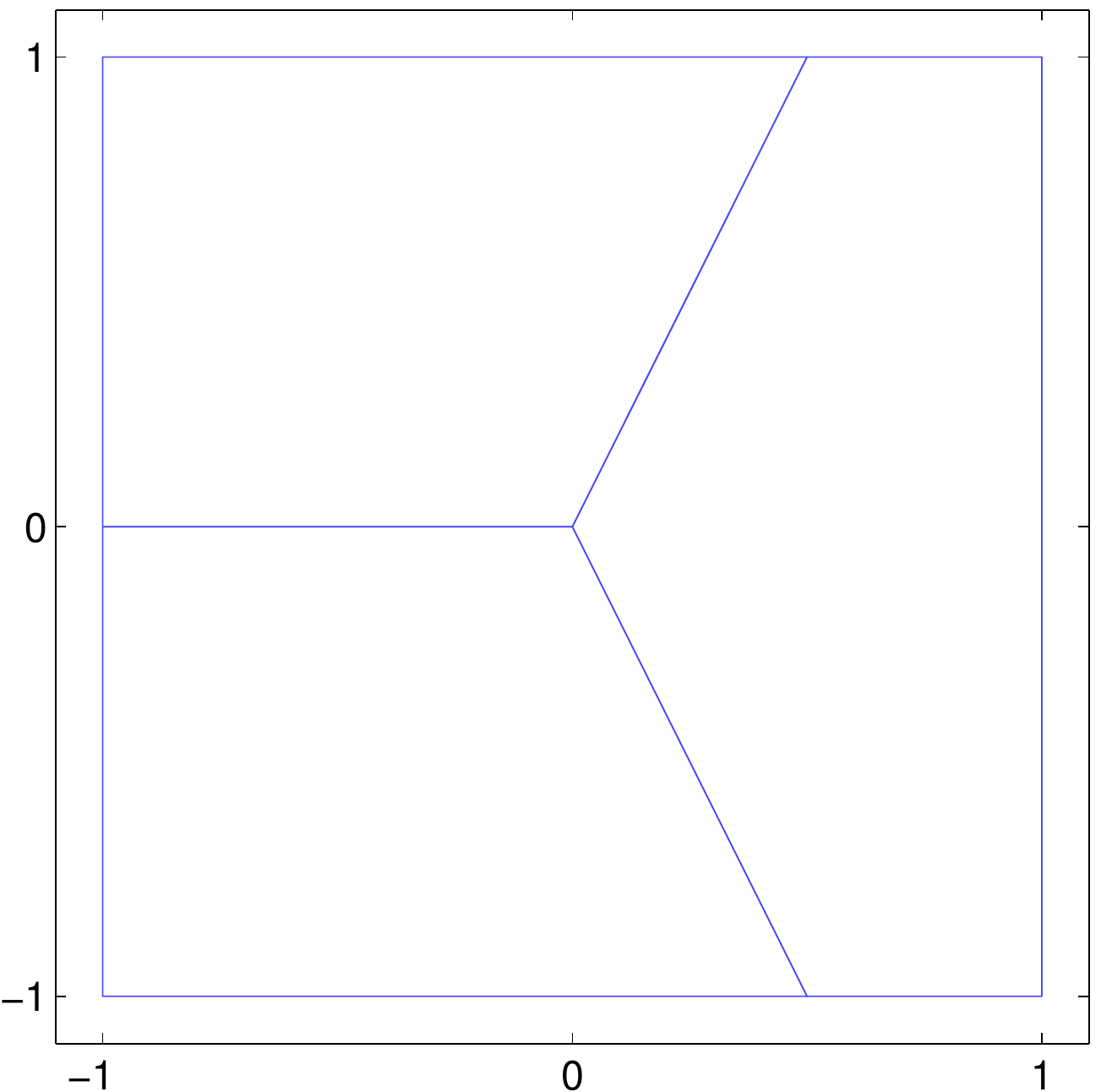}
\includegraphics[height=3.6cm]{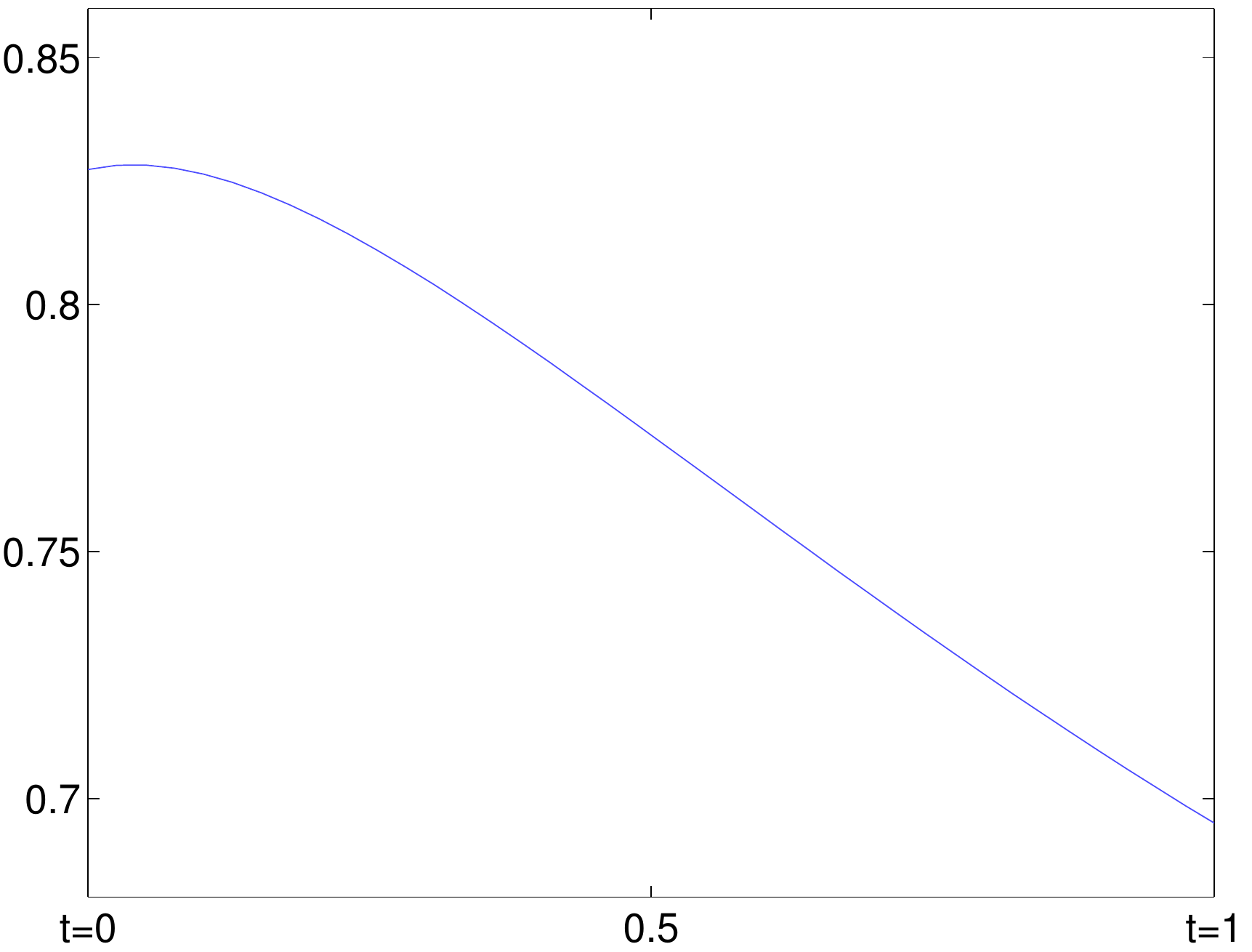}
}
\caption{We consider a point set $P = \{q, p_\pm\}$, with $q = (2,0)$
  and $p_\pm = (0,\pm 1)$, and a function $\phi_t$ which linearly
  interpolates between $\phi_0 = \chi_{\{q\}}$ and $\phi_1 = 0$.
  (Left) Laguerre cells $\Lag_P^{\phi_t}(p)$ intersected with the
  square $[-1,1]^2$ at $t=0$. (Middle) Laguerre cells at
  $t=1$. (Right) Graph of the second moment of the measure
  $G_{\phi_t{\#}}^\ac \mu$ as a function of $t$, showing the lack of convexity of a discretized energy. 
  \label{fig:potential-nonconvex}}
\end{figure}

\begin{remark} The convexity of the internal energy \eqref{eq:conv:U}
  also holds when considering the monotone discretization of the
  Monge-Ampère operator introduced by Oberman in
  \cite{oberman2008wide}.
\end{remark}

\begin{remark}
  It seems necessary to consider two notions of push-forward of a
  given measure $\mu$. Indeed, the internal energy of a measure that
  is not absolutely continuous is $+\infty$, so that it only makes
  sense to compute the map $\mU$ on the absolutely continuous measure
  $G_{\phi\#}^\ac \mu$. On the other hand, condition \HE is not
  sufficient to make the potential energy functional $\phi \in
  \mK^G_Y(P) \mapsto \mE(G^\ac_{\phi\#} \mu)$ convex. This can be seen
  on the example given in Figure~\ref{fig:potential-nonconvex}: let
  $Y=[-1,1]^2$ and $P = \{q, p_\pm\}$ with $q = (2,0)$ and $p_\pm =
  (0,\pm 1)$. We let $\phi_t$ be the linear interpolation between
  $\phi_0 := \one_{\{q\}}$, and $\phi_1 = 0$, and We let $\mu = 0.8 \delta_q
  + 0.1 \delta_{p_+} + 0.1 \delta_{p_-}$. The third column of
  Figure~\ref{fig:potential-nonconvex} displays the graph of the
  second moment of the absolutely continuous push-forward, i.e.
\begin{equation}
 t \mapsto \mE(G_{\phi_t\#}^\ac \mu), \hbox{ where } \mE(\nu) = \int_{\Rsp^d} \nr{x}^2 \dd \nu(x), 
\end{equation}
The graph shows that this function is not convex in $t$, even though
$\mE$ is convex under generalized displacement since it satisfies McCann's
condition \HE.
\end{remark}
\begin{remark}
  The two maps considered in the Theorem can be computed more
  explicitely:
\begin{align}
  \mE(G_{\phi\#} \mu) &= \sum_{p \in P} \mu_p V(G_\phi(p)) +
  \sum_{p,q\in P} \mu_p \mu_q W(G_{\phi(p)}, G_{\phi(q)})\\
  \mU(G^\ac_{\phi\#} \mu) &= \sum_{p\in P} U\left(\frac{\mu_p}{\MA_Y[\phi](p)}\right) \MA_Y[\phi](p)
\end{align}
In particular, when $\mU$ is the negative entropy ($U(r) = r \log r$), one has:
\begin{equation}
  \mU(G^\ac_{\phi\#} \mu) = - \sum_{p\in P} \mu_p\log(\MA_Y[\phi](p)).
\end{equation}
Consequently the internal energy term plays the role of a barrier for
the constraint set $\mK_Y(P)$, that is: if $\mU(G^\ac_{\phi\#}\mu)$ is
finite, then $\phi$ belongs to the interior of $\mK_Y(P)$. The same
behavior remains true if the function $U$ has super-linear growth at
infinity. This enables us to extend $\mU(G^\ac_{\phi\#}\mu)$ to the
whole space $\Rsp^P$, by setting it to $+\infty$ when
$\MA_Y[\phi](p)=0$ for some $p\in P$.
\end{remark}


\section{A convergence theorem}
\label{sec:convergence}

Let $X, Y$ be two convex domains in $\Rsp^d$, and $\mu$ be a
probability measure on $X$ which is absolutely continuous with respect
to the Lebesgue measure on $X$, and whose density $\rho$ is bounded
from above and below: $\rho \in [r,1/r]$, with $r>0$. We are
interested in the minimization problem
\begin{align}
  &\min_{\nu\in \Prob(Y)} \mF(\nu) =   \min \{\mF(\restr{\nabla \phi}{\#} \mu); \phi \in \mK_Y\}, \label{eq:Prob} \\
&\hbox{where }\mF(\nu) := \Wass_2^2(\mu,\nu) + \mE(\nu) + \mU(\nu),
\end{align}
and where the terms of the functional $\mF$ satisfy the following
assumptions:
\begin{itemize}
\item[(C1)] the energy $\mE$ (resp $\mU$) is weakly continuous  (resp. lower semicontinuous) on $\Prob(Y)$;
\item[(C2)] $\mU$ is an internal energy, defined as in \eqref{eq:U},
  where the integrand $U:\Rsp\to\Rsp$ is convex, $U(0)=0$ and $U$ has superlinear growth at infinity i.e. $\lim_{s\to \infty} s^{-1}U(s)=+\infty$. 
\end{itemize}


\begin{remark} Note that the condition (C2) is different from McCann's
  condition \HU for the displacement convexity of an internal
  energy. Among the internal energies that satisfy both McCann's
  conditions and (C1)--(C2), one can cite those that occur in the
  gradient flow formulation of the heat equation, where $U(r) = r\log
  r$, and of the porous medium equation, for which $U(r)
  =\frac{1}{m-1} r^m,$ with $m>1$. The superlinear growth assumption
  in (C2) ensures that the internal energy acts as a barrier for the
  convexity constraint in the approximated problem \eqref{eq:Probn}.
\end{remark}

\begin{theorem}[$\Gamma$-convergence]
  Assume (C1)-- (C2). Let $\mu_n$ be a sequence of probability measures
  supported on finite subsets $P_n\subseteq X$, converging weakly to
  the probability density $\rho$, and consider the discretized problem
\begin{equation}
\label{eq:Probn}
 \min_{(\phi,G_\phi) \in \mK^G_Y(P_n)} \Wass_2^2(\mu_n,G_{\phi_n\#} \mu_n) + \mE(G_{\phi_n\#} \mu_n) + \mU(G_{\phi_n\#}^\ac \mu_n).
\end{equation}
Then, there exists a minimizer  $\phi_n$ of \eqref{eq:Probn}. Moreover, the sequence of
absolutely continuous measure $\sigma_n := G_{\phi_n\#}^\ac \mu_n$ is a
minimizing sequence for the problem \eqref{eq:Prob}. If $\mF$ has a
unique minimizer $\nu$ on $\Prob(Y)$, then $\sigma_n$ converges weakly to
$\nu$.
\end{theorem}

\begin{step}
There exists a minimizer to \eqref{eq:Probn}.
\end{step}

\begin{proof}
Let $(\phi_n^k)_k$ be a minimizing sequence (which we can normalize by imposing $\phi_n^k (p)=0$ at a fixed $p\in P_n$). Since $Y$ is bounded, we may assume  that, up to some not relabeled subsequences $\phi_n^k$  and $G_{\phi_n^k}$ converge to some $(\phi_n, G_{\phi_n})$. We can also assume that $\hat{\phi}_n^k
  := [\phi_n^k]_{\mK_Y}$ converges uniformly to $\hat{\phi}_n=[\phi_n]_{\mK_Y}$. The convergence in the Wasserstein term and in $\mE$ is then obvious, it remains to prove a liminf inequality for the discretized internal energy. First note that thanks to (C2), we also have that there is a $\nu>0$ such that $\MA_Y[\phi_n^k](p)\ge \nu$ for every $k$ and every $p\in P_n$. Then observe that the internal energy can be written as
\[ \mU(G_{\phi_n^k\#}^\ac \mu_n):=\sum_{p\in P_n} F(p, \MA_Y[\phi_n^k](p)), \; F(p,t):=  t U\left(\frac{\mu_p}{t}\right)\]
so that $F(p,.)$ is nonincreasing thanks to (C2). It is then enough to prove that for every $p\in P_n$ one has:
\begin{equation}
 \limsup_k   \MA_Y[\phi_n^k](p)=  \limsup_k  \Haus^d (\partial \hat{\phi}_n^k(p))  \le    \Haus^d (\partial \hat{\phi}_n(p))
\end{equation} 
but the latter inequality follows at once from Fatou's Lemma and the fact that if $y$ belongs to  $\partial \hat{\phi}_n^k(p)$ for infinitely many $k$ then it also necessarily belongs to $\partial \hat{\phi}_n(p)$. This proves that $\phi_n$ solves \eqref{eq:Probn}. 
\end{proof}

Let $m$ and $m_n$ be the minima of \eqref{eq:Prob} and \eqref{eq:Probn}
respectively. Our goal now is to show that $\lim_{n\to\infty} m_n = m$. In
order to simplify the proof, we will keep the same notation for an
absolutely continuous probability measure and its density.

\begin{step}
$\lim\inf_{n\to\infty} m_n \geq m$
\end{step}
\begin{proof} For every $n$, let $\phi_n \in \mK_Y(P_n)$ be a
  minimizer of the discretized problem \eqref{eq:Probn}. By
  compactness of the set $\mK_Y$ (up to an additive constant), and
  taking a subsequence if necessary, we can assume that $\hat{\phi}_n
  := [\phi_n]_{\mK_Y}$ converges uniformly to a function $\phi$ in
  $\mK_Y$. We can also assume that both sequence of measures $\sigma_n
  := G_{\phi_n\#}^\ac \mu_n$ and $\nu_n := G_{\phi_n\#} \mu_n$
  converge to two measures $\sigma$ and $\nu$ for the Wasserstein
  distance. The difficulty is to show that these two measures $\nu$
  and $\sigma$ must coincide. Indeed, let $\pi_n$ (resp. $\pi_n'$) be
   optimal transport plans between $\mu_n$ and $\nu_n$
  (resp. $\mu_n$ and $\sigma_n$). Taking subsequences if necessary,
  these optimal transport plans converge to two transport plans $\pi$
  (resp. $\pi'$) between $\rho$ and $\nu$ (resp. $\rho$ and $\sigma$)
  that are supported on the graph of the gradient of $\phi$. Since the
  first marginal of $\pi$ and $\pi'$ coincide, one must have
  $\pi=\pi'$ and therefore $\nu = \sigma$.  The result then follows
  from the weak lower semicontinuity of $\mU$, and the continuity of
  $(\mu,\nu) \mapsto \Wass_2^2(\mu,\nu) + \mE(\nu)$.
\end{proof}

We now proceed to the proof that $\lim\sup_{n\to\infty} m_n \leq m$.
Our first step is to show that probability measures with a smooth
density bounded from below and above are dense in energy. More
precisely, we have:

\begin{step} 
$m = \min_{\eps > 0} \min \{\mF(\sigma); \sigma \in \Prob^\ac(Y) \cap \Class^0(Y), \eps \leq \sigma \leq 1/\eps\}$
\end{step}

\begin{proof} Let $\sigma$ be a probability density on $Y$ such that
  $\mF(\sigma) < +\infty$. Then, according to Corollary~1.4.3 in
  \cite{agueh2002existence}, there exists a sequence of probability
  densities $\sigma_n$ on $Y$ that satisfy the three properties:
\begin{itemize}
\item[(a)] For every $n>0$, $\sigma_n$ is bounded from above and below:
$$ 0 < \inf_{y\in Y} \sigma_n(y)  < \sup_{y\in Y} \sigma_n(y) < +\infty; $$
\item[(b)] $\sigma_n$ converges to $\sigma$ in $\LL^1(Y)$; 
\item[(c)] $\mU(\sigma_n) \leq \mU(\sigma)$.
\end{itemize} Moreover the proof of Corollary~1.4.3 in
\cite{agueh2002existence} can be modified by taking a smooth
convolution operator so as to ensure that each $\sigma_n$ is
continuous on $Y$. Our task is then to show that
$$ \lim\inf_{n\to\infty} \mF(\sigma_n) \leq \mF(\sigma), $$
where $\mF(\sigma) = \Wass^2_2(\mu,\sigma) + \mE(\sigma) +
\mU(\sigma)$. Thanks to (C1), and thanks to the Wasserstein continuity
of the terms $\sigma \mapsto \Wass^2_2(\mu,\sigma) + \mE(\sigma)$, we
only need to show that $\sigma_n$ converges to $\sigma$ in the
Wasserstein sense. This follows from the easy inequality
\begin{equation*}
  \Wass^2_2(\sigma,\sigma') \leq \nr{\sigma - \sigma'}_{\LL^1(Y)} \diam(Y)^2. 
  \qedhere
\end{equation*}
\end{proof}




\begin{step} Let $\sigma \in \Prob(Y) \cap \Class^0(Y)$, with
  $\eps\leq\sigma \leq 1/\eps$. Then, for every $n\geq 0$, there
  exists a convex interpolate $\phi_n \in \mK_Y(P_n)$ such that
\begin{equation} \forall p\in P_n,~ \sigma(\partial [\phi_n]_{\mK_Y}(p)) = \mu_n(\{p\}).
\end{equation}
\end{step}

\begin{proof}
By Breniers' theorem, there is  a convex potential $\psi_n$  on $Y$ such that  $\nabla \psi_n \# \sigma=\mu_n$, so that $\phi_n:=\psi_n^*$ has the desired property. 
\end{proof}

\begin{step} Assuming that the functions $\phi_n$ in $\mK_Y(P_n)$ are
  constructed as above, we can bound the diameter of their
  subdifferentials:
\begin{equation}
\lim_{n\to\infty} \max_{p\in P_n} \diam(\partial [\phi_n]_{\mK_Y}(p)) = 0.
\label{eq:diam}
\end{equation}
\end{step}

\begin{proof} Let $\hat{\phi} \in \mK_Y$ be a potential for the
  quadratic optimal transport problem between $\rho$ and $\sigma$. Let
  $\hat{\phi}_n := [\phi_n]_{\mK_Y}$ and $\psi = \hat{\phi}^*$ and
  $\psi_n = \hat{\phi}_n^*$. First, we add a constant to $\phi$ and
  $\phi_n$ such that the integral of $\psi$ and $\psi_n$ over $\sigma$
  is zero,
$$\int_Y \psi(y) \sigma(y) \dd y = \int_Y \psi_n(y) \sigma(y) \dd y = 0.$$
Poincaré's inequality on $Y$ with density $\sigma$ gives us
\begin{align*}
\int_{Y} \abs{\psi_n(y)  - \psi(y)}^2 \sigma(y) \dd y 
&\leq \const(p,Y,\sigma) \int_{Y} \nr{\nabla \psi_n - \nabla\psi(y)}^2
\sigma(y) \dd y,
\end{align*}
and the weak continuity of optimal transport plans then ensures that
the right-hand term converges to zero. Noting that $\psi_n$ and $\psi$
are convex on $Y$ and have a bounded Lipschitz constant, because the
gradients $\nabla\psi,\nabla \psi_n$ belong to $X$, this implies that
$\psi_n$ converge uniformly to $\psi$. Taking the Legendre transform,
this shows that $\hat{\phi}_n$ converges uniformly to $\hat{\phi}$ on
the compact domain $X$.

We now prove \eqref{eq:diam} by contradiction, and we assume that
there exists a positive constant $r$, and a sequence of points
$(p_n)$, with $p_n \in P_n$ and such that there exists $y_n, y_n'
\in \partial\hat{\phi}_n$ with $\nr{y_n - y_n'} \geq r$. By
compactness, and taking subsequences if necessary, we can assume that
$p_n$ converges to a point $p$ in $X$ and that the sequences $(y_n)$
and $(y_n')$ converge to two points $y,y'$ in $Y$ with $\nr{y-y'} \geq
r$. Since the point $y_n$ belongs to $\partial\hat{\phi}_n(p_n)$, one
has:
$$ \forall x \in X,~ \hat{\phi}_n(x) \geq \hat{\phi}_n(p_n)+\sca{y_n}{x-p}.$$
Taking the limit as $n$ goes to $\infty$, this shows us that $y$ (and
similarly $y'$) belongs to $\partial\hat{\phi}(p)$, so that
$\diam(\partial\hat{\phi}(p)) \geq r$. The contradiction then follows
from Caffarelli's regularity result \cite{caffarelli1992boundary}: under the  assumptions on the
supports and on the densities, the map $\hat{\phi}$ is
$\Class^{1,\beta}$ up to the boundary of $X$. In particular, the
subdifferential of $\hat{\phi}$ must be a singleton at every point of
$X$, thus contradicting the lower bound on its radius.
\end{proof}

\begin{step} Let $\sigma_n := G_{\phi_n\#}^\ac \mu_n$ and $\nu_n :=
  G_{\phi_n\#} \mu_n$, where $\phi_n$ is defined above. Then,
\begin{align}
&\lim_{n\to\infty} \nr{\sigma_n - \sigma}_{\LL^\infty(Y)} = 0.\\
&\lim_{n\to\infty} \Wass_2(\nu_n,\sigma) = 0. \label{eq:nunsigma}
\end{align}
\end{step}

\begin{proof} First, note that since $\sigma$ is continuous on a
  compact set, it is also uniformly continuous. For any $\delta > 0$,
  there exists $\eps>0$ such that $\nr{x-y}\leq \eps$ implies
  $\abs{\sigma(x)-\sigma(y)}\leq \delta$. Using
  Equation~\eqref{eq:diam}, for $n$ large enough, the sets $V_p
  := \partial[\phi_n]_{\mK_Y}(p)$ have diameter bounded by $\eps$ for
  all point $p$ in $P_n$. By definition, the density $\sigma_n$ is
  equal to
\begin{equation}
  \sigma_n = \sum_{p\in P} \hat{\sigma}_p \chi_{V_p} \hbox{ with } \hat{\sigma}_p := \frac{1}{\Haus^d(V_p)}
  \int_{V_p} \sigma(x) \dd x.
  \label{eq:sigman}
\end{equation}
By the uniform continuity property, on every cell $V_p$ one has
$\abs{\sigma(x) - \hat{\sigma}_p} \leq \delta$, thus proving
$\nr{\sigma_n - \sigma}_{\LL^\infty(Y)} \leq \delta$ for $n$ large
enough. This implies that $\sigma_n$ converges to $\sigma$ uniformly,
and a fortiori that $\lim_{n\to\infty} \Wass_2(\sigma_n,\sigma) = 0$. Then, 
\begin{equation}
\Wass_2(\nu_n,\sigma) \leq \Wass_2(\nu_n,\sigma_n) + \Wass_2(\sigma_n,\sigma).\label{eq:triangle}
\end{equation}
Moreover, one can bound the Wasserstein distance explicitely between
$\sigma_n$ and $\nu_n$ by considering the obvious transport plan on
each of the subdifferentials $(\partial[\phi_n]_{\mK_Y}(p))_{p\in P_n}$:
\begin{equation}
\Wass^2_2(\nu_n,\sigma_n) \leq \sum_{p\in P_n} \diam(\partial[\phi_n]_{\mK_Y}(p)) \mu_p \leq \max_{p\in P_n} \diam(\partial[\phi_n]_{\mK_Y}(p)).
\label{eq:boundnusigma}
\end{equation}
The second statement \eqref{eq:nunsigma}  follows from Eqs.
\eqref{eq:triangle}, \eqref{eq:boundnusigma} and \eqref{eq:diam}.
\end{proof}

\begin{step} $\lim_{n\to\infty} \Wass_2^2(\mu_n,\nu_n) + \mE(\nu_n) +
  \mU(\sigma_n) = \Wass_2^2(\mu,\sigma) + \mE(\sigma) + \mU(\sigma)$
\end{step}

\begin{proof} The convergence of the first two terms follows from the
  Wasserstein continuity of the map $(\mu,\nu) \in \Prob(Y) \mapsto
  \Wass_2^2(\mu,\nu) + \mE(\nu)$. In order to deal with the third term,
  we will assume that $n$ is large enough, so that the densities
  $\sigma, \sigma_n$ belong to the segment $S_{\eps/2} = [\eps/2,
  2/\eps]$. The integrand $U$ of the internal energy is convex on
  $\Rsp$, and therefore Lipschitz with constant $L$ on $S_{\eps/2}$,
  so that
  \begin{align*}
\abs{\int_Y U(\sigma_n) \dd x - \int_Y U(\sigma(x)) \dd x} &\leq 
\int_Y  \abs{U(\sigma_n) - U(\sigma(x))} \dd x \\
&\leq L \nr{\sigma_n - \sigma}_{\LL^\infty(Y)}.\qedhere
\end{align*}
\end{proof}

\section{Numerical results}
\label{sec:numerical}

\subsection{Computation of the Monge-Ampère operator}
In this paragraph we explain how to evaluate the discretized internal
energy of $G_{\phi\#}^\ac\mu_P$, where $\phi$ is a discrete convex
function in $\mK_Y(P)$, and $Y$ is a polygon in the euclidean
plane. 
Thanks to the equation
\begin{equation}
  \mU(G_{\phi\#}^\ac\mu_P) 
  = \sum_{p\in P} U\left(\mu_p/\MA_Y[\phi](p)\right)\MA_Y[\phi](p),
\end{equation}
one can see that the internal energy and its first and second
derivatives can be easily computed if one knows how to evaluate the
discrete Monge-Ampère operator and its derivatives with respect to
$\phi$. Our assumptions for performing this computation will be the
following:
\begin{itemize}
\item[(G1)] the domain $Y$ is a convex polygon and its boundary can be
  decomposed as a finite union of segments $S = \{s_1,\hdots,s_k\}$
\item[(G2)] the points in $P$ are in generic position, i.e. (a) there
  does not exist a triple of collinear points in $P$ and (b) for any
  pair $p,q$ of distinct points in $P$, there is no segment $s$ in $S$
  which is collinear to the bisector of $[pq]$.
\end{itemize}
The Jacobian matrix of the discrete Monge-Ampère operator is a square
matrix denoted $(\JMA_Y[\phi])_{p,q\in P}$, while its Hessian is a
$3$-tensor denoted $(\HMA_Y[\phi])_{p,q,r\in P}$. The entries of this
matrix and tensor are given by the formulas
\begin{align}
\JMA_Y[\phi]_{pq} &:= \frac{\partial\MA_Y[\phi](p)}{\partial \one_q},\\
\HMA_Y[\phi]_{pqr} &:= \frac{\partial^2\MA_Y[\phi](p)}{\partial \one_r\partial \one_q},
\end{align}
where $\one_p$ denotes the indicator function of a point $p$ in
$P$. The goal of the remaining of this section is to show how the
computation of the Jacobian matrix and the Hessian tensor are related
to a triangulation which is defined from the Laguerre cells by
duality.
\subsubsection{Abstract dual triangulation}
  Given any function $\phi$ on $P$, we introduce a notation for the
  intersection of the Laguerre cell of $P$ with $Y$, and we extend
  this notation to handle boundary segments as well. More precisely,
  we set:
\begin{equation}
\begin{aligned}
\forall p\in P,~  V^\phi(p) &:= \Lag_P^\phi(p) \cap Y,\\
\forall s\in S,~  V^\phi(s) &:= s.
\end{aligned}
\end{equation}
We also introduce a notation for the finite intersections of these cells:
\begin{equation}
\forall p_1,\hdots,p_s \in P\cup S,~ V^\phi(p_1\hdots p_s) := V^\phi(p_1)\cap ...\cap V^\phi(p_s)
\end{equation}
The decomposition of $Y$ given by the cells $V^\phi(p)$ induces an
abstract dual triangulation $T^\phi$ of the set $P\cup S$, whose
triangles and edges are characterized by:
\begin{itemize}
\item[(i)] a pair $(p,q)$ in $P\cup S$ is an edge of $T^\phi$ iff
  $V^\phi(pq) \neq \emptyset$;
\item[(ii)] a triplet $(p,q,r)$ in $P\cup S$ is a \emph{triangle} of
  $T^\phi$ iff $V^\phi(pqr) \neq \emptyset$.
\end{itemize}
An example of such an abstract dual triangulation is displayed in
Figure~\ref{fig:dual-triangulation}.

The construction of this triangulation can be performed in time
$\mathrm{O}(N\log N + k)$, where $N$ is the number of points and $k$
is the number of segments in the boundary of $Y$. The construction
works by adapting the regular triangulation of the point set, which is
the triangulation obtained when $Y=\Rsp^d$, and for which there exists
many algorithms, see e.g. \cite{cgal}.


\subsubsection*{Jacobian of the Monge-Ampère operator} By
Lemma~\ref{lem:convex}, for any point $p$ in $P$, the function
$\phi\mapsto \MA_Y[\phi](p)$ is log-concave on the set
$\mK_Y(P)$. This function is therefore twice differentiable almost
everywhere on the interior of $\mK_Y(P)$, using Alexandrov's
theorem. The first derivatives of the Monge-Ampère operator is easy to
compute, and involves boundary terms: two points $p,q$ in $P$,
\begin{align}
  \JMA_Y[\phi]_{pq} &=
  \frac{\Haus^1(V^\phi(pq))}{\nr{p-q}} \hbox{ if } q\neq p \label{d:qp}\\
  \JMA_Y[\phi]_{pq} &= - \sum_{\substack{q \in P\\ (qp) \in T^\phi}}
  \frac{\Haus^1(V^\phi(pq))}{\nr{p-q}} \label{d:pp}
\end{align}
Note that every non-zero element in the square matrix corresponds to
an edge in the dual triangulation $T^\phi$.

\subsubsection{Hessian of the Monge-Ampère operator} We will not
include the computation of the second order derivatives, but we will
sketch how it can be performed using the triangulation
$T^\phi$. First, we remark that thanks to our genericity assumption,
for every triangle $pqr$ of $T^\phi$, the set $V^\phi(pqr)$ consists
of a single point, which we also denote $V^\phi(pqr)$. For any edge
$pq$ in the triangulation $T^\phi$, where $p,q$ are two points in $P$,
the intersection $V^\phi(pq)= V^\phi(q)\cap V^\phi(q)$ is a segment
$[x,y]$. The endpoint $x$ of this segment needs to be contained in a
third cell $V^\phi(r)$ for a certain element $r$ of $P\cup S \setminus
\{p,q\}$, so that $x = V^\phi(pqr)$. Similarly, there exists $r'$ in
$P\cup S \setminus \{p,q\}$ such that $y = V^\phi(pqr')$. One can
therefore rewrite the length of $V^\phi(pq)$ as
\begin{equation}
\Haus^1(V^\phi(pq)) = \nr{V^\phi(pqr) - V^\phi(pqr')}.
\label{d:length}
\end{equation}
The expression of the Hessian can be deduced from Equations
\eqref{d:qp}--\eqref{d:pp} and \eqref{d:length}, and from an explicit
computation for the point $V^\phi(pqr)$. Moreover, to each nonzero
element of the Hessian one can associate a point, an edge or a
triangle in the triangulation $T^\phi$. More precisely:
\begin{align*}
\HMA_Y[\phi]_{pqr} \neq 0 \Longrightarrow p=q=r
&\hbox{ or } (p=q \hbox{ and } (pr) \hbox{ is an edge of } T^\phi) \\
&\hbox{ or } (pqr) \hbox{ is a triangle of } T^\phi.
\end{align*}
In particular, the total number of non-zero elements of the tensor
$\HMA_Y[\phi]$ is at most proportional to the number $\abs{P}$ of
points plus the number $\abs{S}$ of segments.
\begin{figure}[t] { \centering
\includegraphics[height=4.5cm]{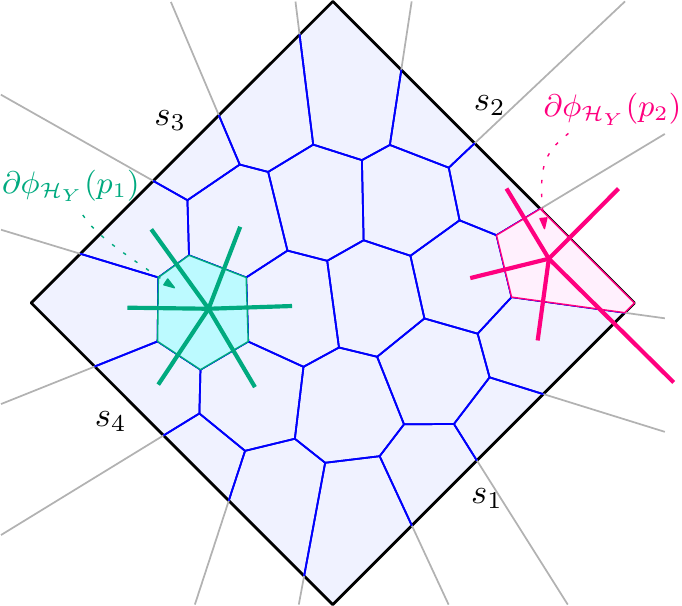}\,
\includegraphics[height=4.5cm]{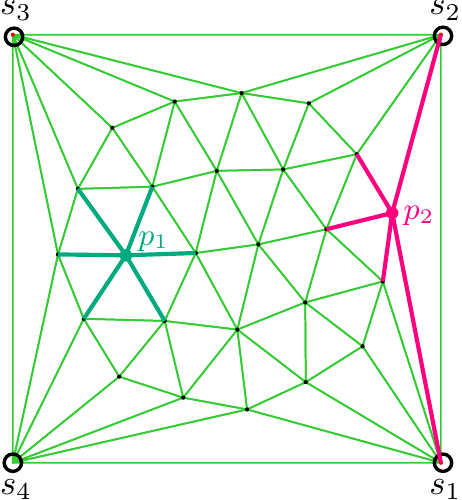}
}
\caption{On the left, the intersection of power cells with a convex polygon (in red), and on the right, the dual triangulation.
  \label{fig:dual-triangulation}}
\end{figure}

\subsection{Non-linear diffusion on point clouds} The first
application is non-linear diffusion in a bounded convex domain $X$ in
the plane. We are interested in the following PDE, where the parameter
$m$ is chosen in $[1-1/d,+\infty)$. A numerical application is
displayed on Figure~\ref{fig:nl} .
\begin{equation}
\left\{\begin{aligned}
\frac{\partial \rho}{\partial t} &=  \Delta \rho^m
 &\hbox{on } X \\
\nabla \rho  &\perp \mathbf{n}_X &\hbox{on } \partial X
\end{aligned}
\right.
\end{equation}
When $m=1$, this PDE is the classical heat equation with Neumann
boundary conditions. When $m<1$, this PDE provides a model of fast
dixffusion, while for $m>1$ it is a model for the evolution of gases in
a porous medium. Otto \cite{otto2001geometry} reinterpreted this PDE
as a gradient flow in the Wasserstein space for the internal energy
\begin{equation}
  \mU_m(\mu) = \left\{\begin{aligned}
      &\int_{\Rsp^d} U_m(\rho(x))  \dd x \hbox{ if } \mu \ll \Haus^d, \rho := \frac{\dd \mu}{\dd\Haus^d}, \\
      &+\infty \hbox{ if not.}
\end{aligned}
\right.,
\label{eq:Um}
\end{equation}
where $U_m(r) = \frac{r^m(x)}{m-1}$ when $m\neq 1$ and $U_1(r) = r\log
r$. A time-discretization of this gradient-flow model can be defined
using the Jordan-Kinderlehrer-Otto scheme: given a timestep $\tau >
0$ and a probability measure $\mu_0$ supported on $X$, one defines a
sequence of probability measures $(\mu_k)_{k\geq 1}$ recursively
\begin{equation}
\mu_{k+1} = \arg\min_{\mu \in \Prob(X)} \Wass_2^2(\mu_k,\mu) + \mU_m(\mu)
\label{eq:grad:nld}
\end{equation}
The energies involved in this optimization problem satisfy McCann's
assumption for displacement convexity, and our discrete framework is
therefore able to provide a discretization in space of Equation
\eqref{eq:grad:nld} as a convex optimization problem. We use this
discretization in order to construct the non-linear diffusion for a
finite point set $P_0$ contained in the convex domain $X$. Note that
for this experiment, we do not use the formulation involving the space
of convex interpolates with gradient, $\mK_X^G(P)$. For every function
$\phi$ in $\mK_Y(P)$, and every point $p$ in $P$, we select
explicitely a subgradient in the subdifferential $\partial
\phi_{\mK_Y}(p)$ by taking its Steiner point \cite{schneider1993convex}.

We start
with $\mu_0 = \sum_{p \in P} \delta_p/\abs{P}$ the uniform measure on
the set $P_0$, and we define recursively

\begin{equation}
\left\{
\begin{aligned}
&\phi_k = \arg\min \left\{ \frac{1}{2\tau} \Wass^2_2(\mu_k,G_{\phi\#}\mu_k) + \mU(G_{\phi\#}^\ac \mu_k); \phi \in \mK_{X}(P_k) \right \} \\
& \mu_{k+1} = G_{\phi_k\#}\mu_k,\\
&P_{k+1} = \spt(\mu_{k+1}) 
\end{aligned}
\right.,
\label{eq:nld:algo}
\end{equation}
where $G_\phi(p)$ is the Steiner point of $\partial \phi_{\mK_Y}(p)$.
This minimization problem is solved using a second-order Newton
method. Note that, as mentioned in the remark following
Theorem~\ref{th:convex}, the internal energy plays the role of a
barrier for the convexity of the discrete function $\phi$. When
second-order methods fail, one could also resort to more robust
first-order methods for the resolution of the optimization problem,
using for instance a projected gradient algorithm.

\begin{figure}
{\centering
\includegraphics[width=.5\textwidth,height=.5\textwidth]{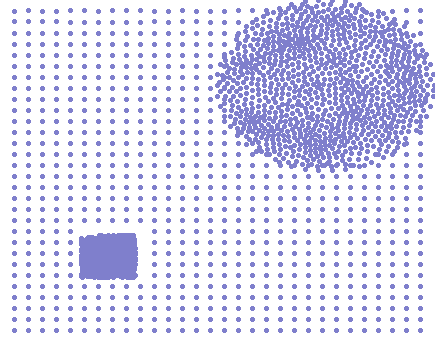}
}

\hspace{-.7cm}
\includegraphics[width=.2\textwidth,height=.2\textwidth]{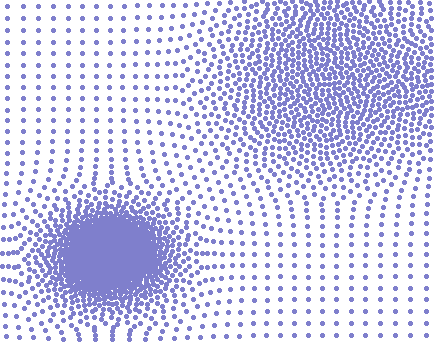}
\includegraphics[width=.2\textwidth,height=.2\textwidth]{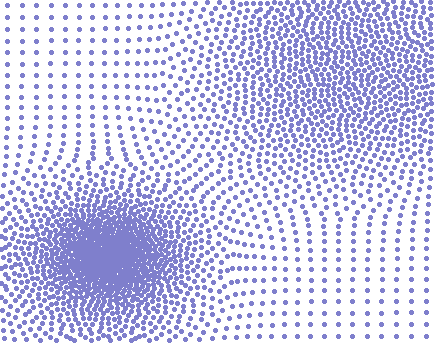}
\includegraphics[width=.2\textwidth,height=.2\textwidth]{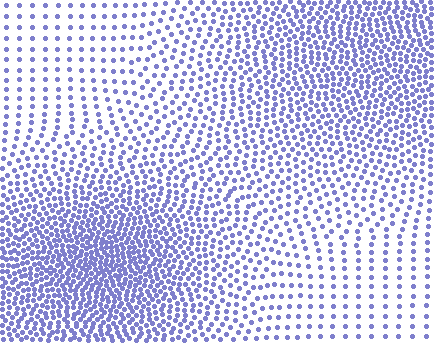}
\includegraphics[width=.2\textwidth,height=.2\textwidth]{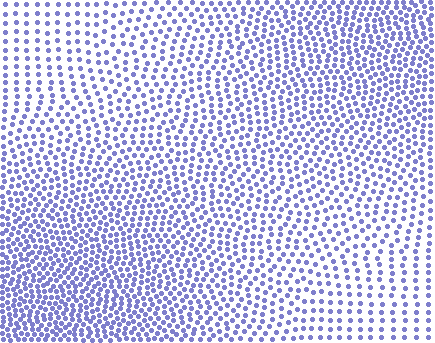}
\includegraphics[width=.2\textwidth,height=.2\textwidth]{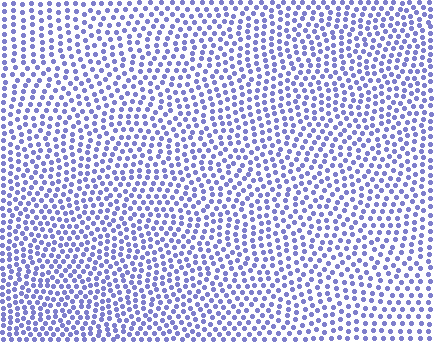}
\\
\hspace{-.7cm}
\includegraphics[width=.2\textwidth,height=.2\textwidth]{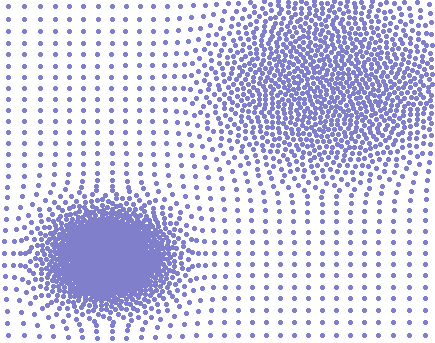}
\includegraphics[width=.2\textwidth,height=.2\textwidth]{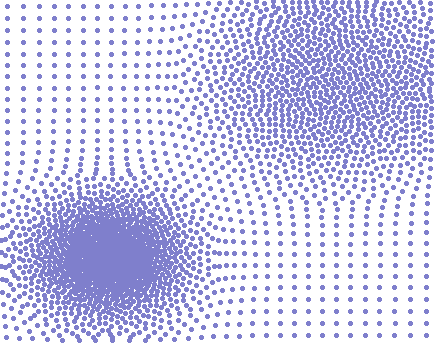}
\includegraphics[width=.2\textwidth,height=.2\textwidth]{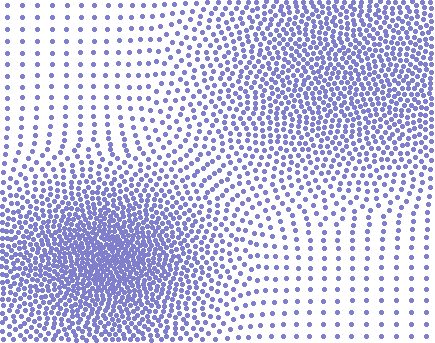}
\includegraphics[width=.2\textwidth,height=.2\textwidth]{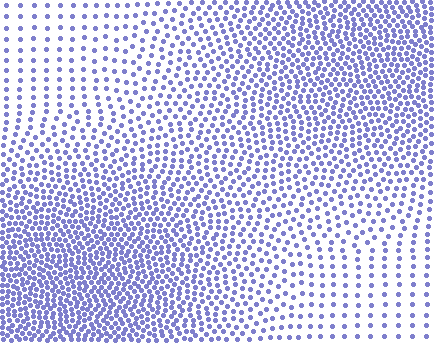}
\includegraphics[width=.2\textwidth,height=.2\textwidth]{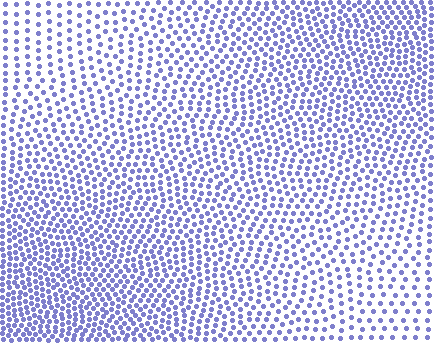}
\\
\hspace{-.7cm}
\includegraphics[width=.2\textwidth,height=.2\textwidth]{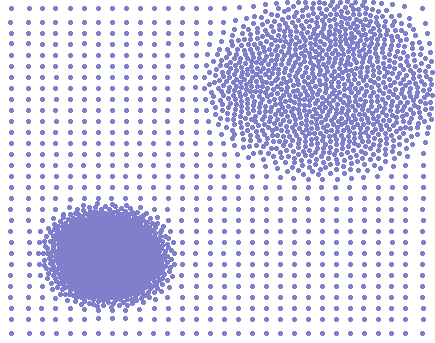}
\includegraphics[width=.2\textwidth,height=.2\textwidth]{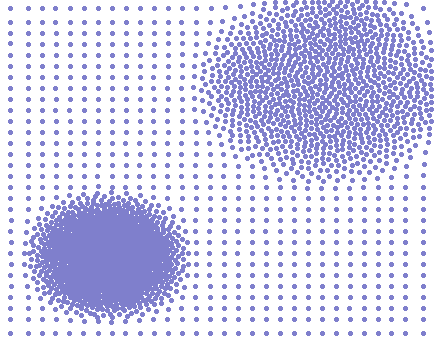}
\includegraphics[width=.2\textwidth,height=.2\textwidth]{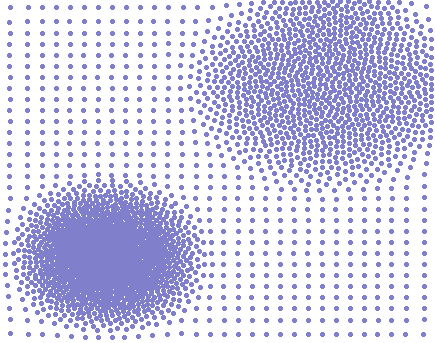}
\includegraphics[width=.2\textwidth,height=.2\textwidth]{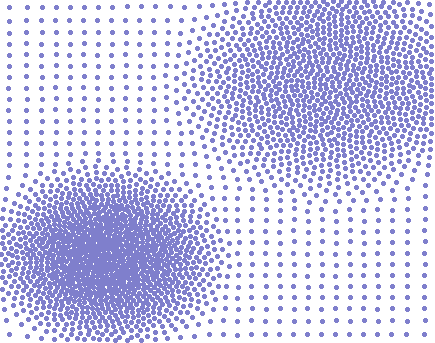}
\includegraphics[width=.2\textwidth,height=.2\textwidth]{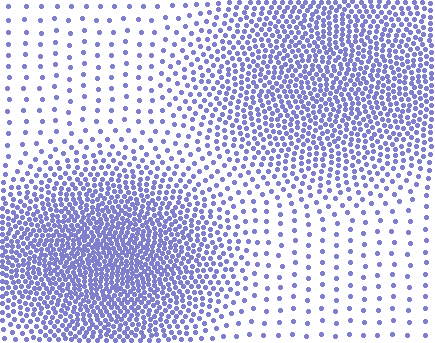}
\\
\hspace{-.7cm}
\includegraphics[width=.2\textwidth,height=.2\textwidth]{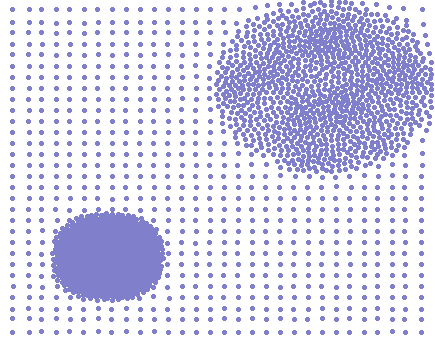}
\includegraphics[width=.2\textwidth,height=.2\textwidth]{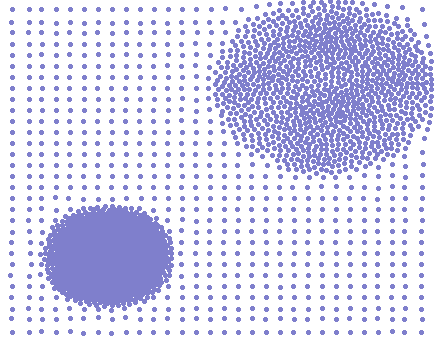}
\includegraphics[width=.2\textwidth,height=.2\textwidth]{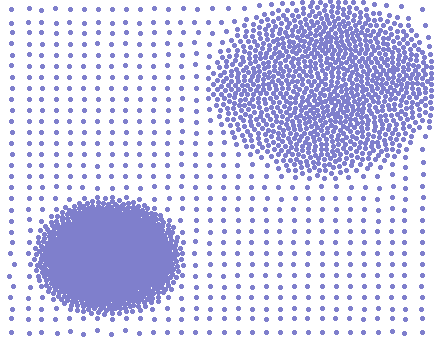}
\includegraphics[width=.2\textwidth,height=.2\textwidth]{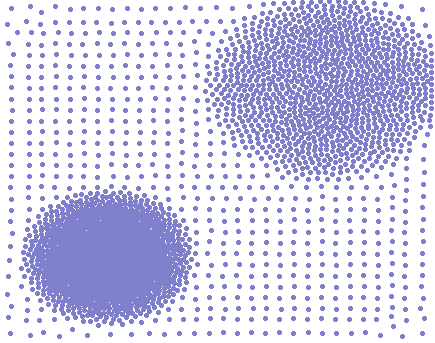}
\includegraphics[width=.2\textwidth,height=.2\textwidth]{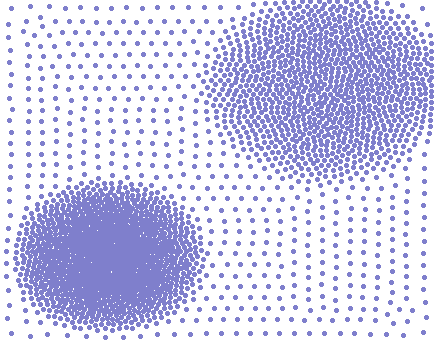}

\caption{(Top) The original point set $P_0$ in the square domain
  $X=[-2,2]^2$ contains 2\,900 points. (Bottom rows) The evolution of
  the point cloud $P_k$ at timesteps $k=5,10,20,40,80$ with
  $\tau=0.01$ is defined using Eq. \eqref{eq:nld:algo} for various
  values of the exponent $m$. From top to bottom, $m$ takes increasing
  values in $\{ 0.6, 1, 2, 3\}$. Note that the case $m=0.6$ is outside
  of the scope of the $\Gamma$-convergence theorem \label{fig:nl}}
\end{figure}

\subsection{Crowd-motion with congestion}
As a second application, we consider the model of crowd motion with
congestion introduced by Maury, Roudneff-Chupin and Santambrogio
\cite{maury2010macroscopic}. The crowd is represented by a probability
density $\mu_0$ on a convex compact subset with nonempty interior $X$
of $\Rsp^2$, which is bounded by a certain constant, which we assume
normalized to one (so that we also naturally assume that
$\Haus^d(X)>1$).  One is also given a potential $V:X\to\Rsp$, which we
assume to be $\lambda$-convex, i.e.  $V(.) + \lambda\nr{.}^2$ is
convex. The evolution of the probability density describing the crowd
is induced by the gradient flow of the potential energy
\begin{equation}
\mE(\mu) = \int_{\Rsp^d} V(x) \dd \mu(x),
\end{equation}
in the Wasserstein space, under the additional constraint that the
density needs to remain bounded by one. We rely again on
time-discretization of this gradient flow using the
Jordan-Kinderlehrer-Otto scheme. This gives us the following
formulation:
\begin{equation}
  \mu_{k+1} = \arg\min_{\mu \in \Prob(X)} \frac{1}{2\tau} \Wass_2^2(\mu_k,\mu) + \mE(\mu) + \mU(\mu),
\end{equation}
where $\mU$ is the indicatrix function of the probability measures
whose density is bounded by one:
\begin{equation}
\mU(\mu) = \left\{\begin{aligned}
&0 \hbox{ if } \mu \ll \Haus^d \hbox{ and } \frac{\dd \mu}{\dd\Haus^d} \leq 1 \\
&+\infty \hbox{ if not.}
\end{aligned}
\right.
\end{equation}
In order to perform numerical simulations, we replace this indicatrix
function by a smooth approximation.
\begin{equation}
  \mU_\alpha(\mu) = \left\{\begin{aligned}
      &\int_{\Rsp^d} \rho^\alpha(x) (-\log(1-\rho(x)^{1/d}) \dd x \hbox{ if } \mu \ll \Haus^d \hbox{ and } \rho := \frac{\dd \mu}{\dd\Haus^d}, \\
      &+\infty \hbox{ if not.}
\end{aligned}
\right.
\label{eq:Ualpha}
\end{equation}
Note that if $\mU_\alpha(\mu)$ is finite, then the density of $\mu$ is
bounded by one almost everywhere. Moreover, we have the following
convexity and $\Gamma$-convergence results:
\begin{proposition}\begin{itemize}
\item[(i)] The energy $\mU_\alpha$ is convex under general displacement.
\item[(ii)] $\mU_\alpha$ $\Gamma$-converges (for the weak convergence of measures on $X$) to $\mU$ as $\alpha$ tends to $+\infty$ ; 
\item[(iii)] $\beta \mU_1$ $\Gamma$-converges to $\mU$ as $\beta$ tends
  to $0$.
\end{itemize}
\end{proposition}

\begin{proof} The proof of (i) uses McCann's theorem: one only needs
  $r^d U(r^{-d})$ to be convex non-increasing and $U(0) = 0$, which
  follows from a simple computation. (ii) The proof of the $\Gamma$-liminf inequality is obvious since $\mU_\alpha \ge \mU$ and $\mU$ is lower semicontinuous. As for the $\Gamma$-limsup inequality, we proceed as follows: we first fix $\mu\in  \Prob(X)$ such that $\mU(\mu)=0$ (otherwise, there is nothing to prove). Let us then fix a set $A\subset X$ such that $\Haus^d(A)>1$ and let $m$ be the uniform probability measure on $A$. For $\eps\in (0,1)$, let us then define $\mu_\eps:=(1-\eps) \mu +\eps m$ so that $\mu_\eps$ has a density bounded by $1-C \eps$ where $C:=1-\frac{1}{\Haus^d(A)}>0$. Letting $\alpha\to \infty$ and setting $\eps_\alpha \sim \alpha^{-1/2}$, one directly checks that $\limsup_{\alpha} \mU_\alpha(\mu_{\eps_\alpha})=O(e^{-\alpha^{1/2}} \log (\alpha)) =0 =\mU(\mu)$ which proves the $\Gamma$-limsup inequality. For (iii), the proof is similar, choosing $\eps_\beta \sim e^{-\beta^{-1/2}}$ as $\beta \to 0$ for the $\Gamma$-limsup inequality.
\end{proof}

\subsubsection*{Numerical result} Figure~\ref{fig:cm} displays a
numerical application, where we compute the Wasserstein gradient flow
of a probability density whose energy is given by
\begin{align}
  &\mF(\rho) = \int_{X} V(x) \rho(x) \dd x + \alpha \mU_1(\rho),\\
&\hbox{where } X=[-2,2]^2,~ \hbox{ and } V(x) = \nr{x - (2,0)}^2 + 5\exp(-5\nr{x}^2/2). \notag
\end{align}
Note that the chosen potential is semi-convex. We track the evolution
of a probability density on a fixed grid, which allows us to use a
simple finite difference scheme to evaluate the gradient of the
transport potential. From one timestep to another, the mass of the
absolutely continuous pushforward of the minimizer is redistributed on
the fixed grid.

\begin{figure}
\centering
\includegraphics[width=.2\textwidth]{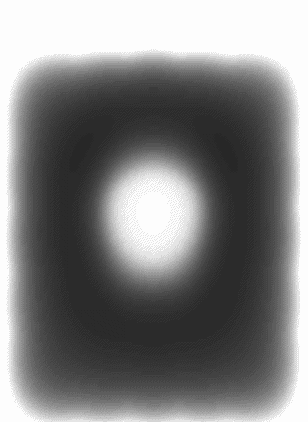}\hspace{-.2cm}
\includegraphics[width=.2\textwidth]{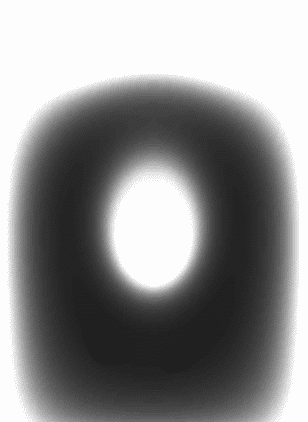}\hspace{-.2cm}
\includegraphics[width=.2\textwidth]{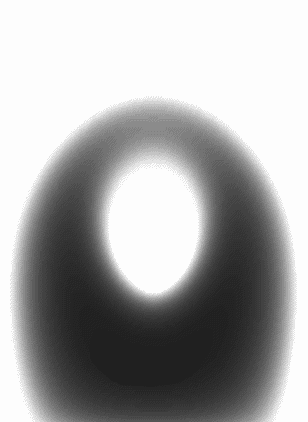}\hspace{-.2cm}
\includegraphics[width=.2\textwidth]{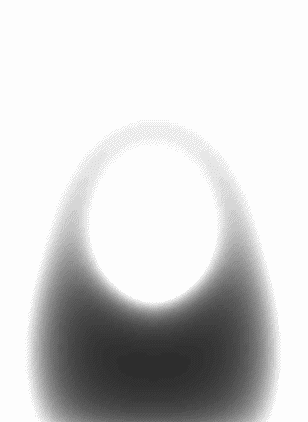}\hspace{-.2cm}
\includegraphics[width=.2\textwidth]{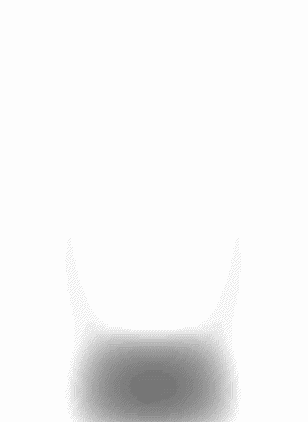}\\
\includegraphics[width=.2\textwidth]{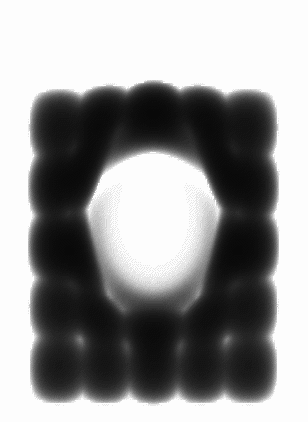}\hspace{-.2cm}
\includegraphics[width=.2\textwidth]{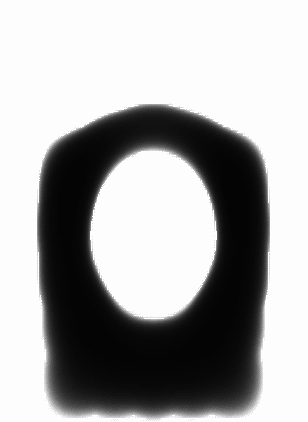}\hspace{-.2cm}
\includegraphics[width=.2\textwidth]{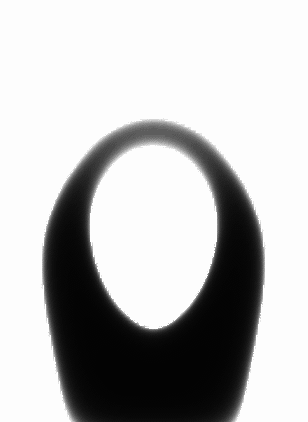}\hspace{-.2cm}
\includegraphics[width=.2\textwidth]{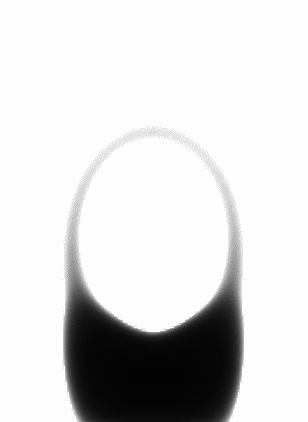}\hspace{-.2cm}
\includegraphics[width=.2\textwidth]{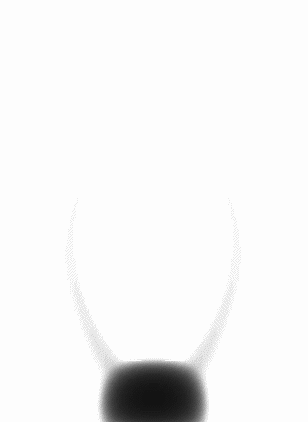}\\
\includegraphics[width=.2\textwidth]{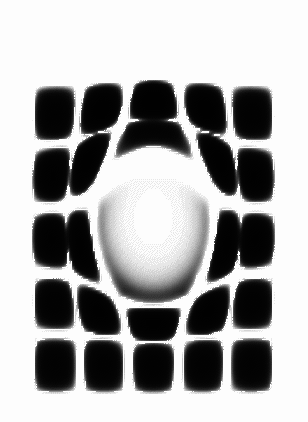}\hspace{-.2cm}
\includegraphics[width=.2\textwidth]{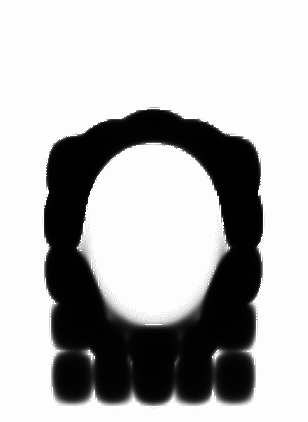}\hspace{-.2cm}
\includegraphics[width=.2\textwidth]{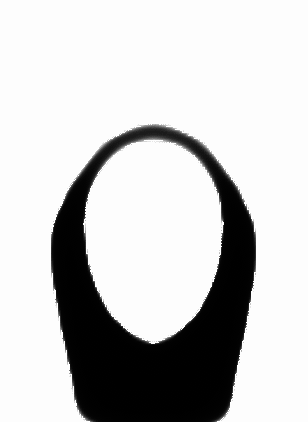}\hspace{-.2cm}
\includegraphics[width=.2\textwidth]{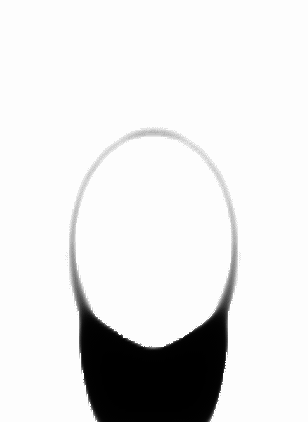}\hspace{-.2cm}
\includegraphics[width=.2\textwidth]{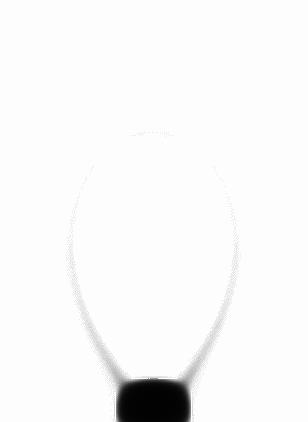}
\caption{Simulation of crowd motion under congestion using the
  gradient flow formulation. The congestion term is given by $\beta
  \mU_1$ (see Equation \eqref{eq:Ualpha}), for various values of
  $\beta$. From top to bottom, $\beta$ is set to $10^{-k}$ for $0\leq
  k\leq 2$. \label{fig:cm}}
\end{figure}

\subsection*{Acknowledgements.} The authors gratefully acknowledge the
support of the French ANR, through the projects ISOTACE
(ANR-12-MONU-0013), OPTIFORM (ANR-12-BS01-0007) and TOMMI
(ANR-11-BSO1-014-01).

\bibliographystyle{amsplain}
\bibliography{jko}

\end{document}